\documentclass{article}

\usepackage{graphicx}

\usepackage{authblk}

\usepackage{tikz}
\usetikzlibrary{arrows,shapes,matrix,
                decorations.pathmorphing,
                shapes.geometric,calc,
                positioning}

\tikzstyle{vertex}=[circle, draw, minimum size=5pt, line width=0.75pt,
					inner sep=0pt, outer sep=0pt]

\tikzstyle{blackV}=[circle, fill=black, minimum size=6pt,
					inner sep=0pt, outer sep=0pt]

\tikzstyle{edge}=[shorten <=1pt, shorten >=1pt, >=stealth, line width=1.1pt]

\tikzstyle{bentE}=[shorten <=1pt, shorten >=1pt, >=stealth, bend right=30,
          bend left=30, line width=1.1pt]

\usepackage{float}

\usepackage{xcolor}

\usepackage{amsmath}
\usepackage{amssymb}

\usepackage{amsthm}

\usepackage[colorlinks=true, linkcolor=blue, citecolor=blue, urlcolor= blue]{hyperref} % Para edición y en línea

\usepackage[capitalize]{cleveref}

\newtheorem{theorem}{Theorem}
\newtheorem{lemma}[theorem]{Lemma}
\newtheorem{corollary}[theorem]{Corollary}

\newtheorem{remark}[theorem]{Remark}
\newtheorem{conjecture}[theorem]{Conjecture}
\newtheorem{problem}[theorem]{Problem}

\newcommand{\set}[1]{\{ #1 \}}
\newcommand{\Deg}[1]{\ensuremath{\textnormal{d}( #1 )}}

\title{Polarity on $H$-split graphs%
\thanks{The authors gratefully acknowledge support from grant DGAPA-PAPIIT
IA101423}}

\author[1]{F.~Esteban~Contreras-Mendoza\thanks{esteban.contreras@ciencias.unam.mx}}
\author[1]{C\'esar~Hern\'andez-Cruz\thanks{chc@ciencias.unam.mx}}

\affil[1]{Facultad de Ciencias, Universidad Nacional Aut\'onoma de M\'exico, Av.
  Universidad 3000, Circuito Exterior S/N, C.P. 04510, Ciudad Universitaria,
  CDMX, M\'exico}

\begin{document}
\date{}

\maketitle
\begin{abstract}
  Given nonnegative integers, $s$ and $k$, an $(s,k)$-polar partition of a graph
  $G$ is a partition $(A,B)$ of $V_G$ such that $G[A]$ and $\overline{G[B]}$ are
  complete multipartite graphs with at most $s$ and $k$ parts, respectively. If
  $s$ or $k$ is replaced by $\infty$, it means that there is no restriction on
  the number of parts of $G[A]$ or $\overline{G[B]}$, respectively. A graph
  admitting a $(1,1)$-polar partition is usually called a split graph.

  In this work, we present some results related to $(s,k)$-polar partitions on
  two graph classes generalizing split graphs. Our main results include
  efficient algorithms to decide whether a graph on these classes admits an
  $(s,k)$-polar partition, as well as upper bounds for the order of minimal
  $(s,k)$-polar obstructions on such graph families for any $s$ and $k$ (even if
  $s$ or $k$ is $\infty$).
\end{abstract}

% % % % % % % % % % % % % % % % % % % % % % % % % % % % % % % % % % % % % % % %
\section{Introduction}
% % % % % % % % % % % % % % % % % % % % % % % % % % % % % % % % % % % % % % % %

All graphs in this work are finite and simple. In general we follow
\cite{bondySpringer2008}, although some notations can differ a little bit. We
use $G + H$ and $G \oplus H$ to denote the disjoint union and the join of the
graphs $G$ and $H$, respectively. Congruently, we use $nG$ to denote the
disjoint union of $n$ copies of a graph $G$. Two subsets $V$ and $W$ of the
vertex set of a graph $G$ are said to be \textit{completely adjacent} if $uw \in
E_G$ for every $u \in U$ and each $w \in W$, and they are \textit{completely
nonadjacent} if $uw \notin E_G$ for every $u \in U$ and each $w \in W$. For a
family $\mathcal{F}$ of graphs, we say that the graph $G$ is
\textit{$\mathcal{F}$-free} if $G$ does not have any graph in $\mathcal{F}$ as
an induced subgraph; if $\mathcal{F} = \{F\}$ we sat that $G$ is $F$-free
instead of $\{F\}$-free. A property of graphs is said to be \textit{hereditary}
if it is closed under induced subgraphs. A \textit{minimal
$\mathcal{P}$-obstruction} for a hereditary property $\mathcal{P}$ of graphs is
a graph that does not have the property $\mathcal{P}$, but such that every
vertex-deleted subgraph does.

A \textit{$(k, \ell)$-coloring} of a graph $G$ is a partition of $V_G$ in $k$
independent sets and $\ell$ cliques. A graph is said to be \textit{$(k,
\ell)$-colorable} if its vertex set admits a $(k,\ell)$-coloring. A
$(k,0)$-coloring of $G$ is called a (proper) \textit{$k$-coloring} of $G$, and
$G$ is said to be \textit{$k$-colorable} if it admits a $k$-coloring. The
minimum integer $k$ such that $G$ admits a $k$-coloring is the \textit{chromatic
number} of $G$, and it is denoted by $\chi(G)$. The minimum integer $\ell$ for
which $G$ has a $(0, \ell)$-coloring is denoted by $\theta(G)$, and it is called
the \textit{clique covering number} of $G$. A \textit{$z$-cocoloring} of $G$ is
any $(k, \ell)$-coloring of $G$ such that $k + \ell = z$. We use $\chi^c(G)$ to
denote the \textit{cochromatic number} of $G$, which is the minimum integer $z$
for which $G$ admits a $z$-cocoloring. A graph $G$ is said to be
\textit{$z$-bicolorable} if, for any integers $k$ and $\ell$ such that $k + \ell
= z$, $G$ is $(k, \ell)$-colorable. The \textit{bichromatic number} of $G$,
denoted $\chi^b(G)$, is the minimum integer $z$ such that $G$ is
$z$-bicolorable. Notice that, for any graph $G$,
\(
  \chi^c(G) \le \min\set{\chi(G), \theta(G)} \le \max\set{\chi(G), \theta(G)}
  \le \chi^b(G).
\)

Given two nonnegative integers $s$ and $k$, a partition $(A, B)$ of the vertex
set of a graph $G$ is called an \textit{$(s,k)$-polar partition} of $G$ if $A$
induces a complete $s$-partite graph and $B$ induces the disjoint union of at
most $k$ complete graphs. An \textit{$(s,k)$-polar graph} is a graph admitting
an $(s,k)$-polar partition. If $s$, $k$, or both of them, are replaced by
$\infty$, it means that there is no restriction on the number of parts of
$G[A]$, $\overline{G[B]}$, or both, respectively. A \textit{polar partition} is
an $(\infty, \infty)$-polar partition. A \textit{monopolar partition}
(respectively, a \textit{unipolar partition}) is a polar partition $(A,B)$ such
that $A$ is an independent set (respectively, a clique). Naturally, graphs
admitting polar, monopolar, or unipolar partitions are called polar, monopolar
and unipolar graphs, respectively.

The problems of deciding whether an arbitrary graph is polar or monopolar are
known to be NP-complete problems
\cite{chernyak1986recognizing,farrugia2004vertex}. In contrast, unipolar graphs
have been shown to be recognizable in $O(n^3)$-time \cite{churchley2014solving,
eschen2014algorithms}, and it was proven in \cite{feder2003list} that
$(s,k)$-polar graphs can be recognized in $O(|V|^{4+2\max\{s,k\}})$-time.

It is known that, for any pair of nonnegative integers $s$ and $k$, there is
just a finite number of minimal $(s,k)$-polar obstructions
\cite{feder2007matrix}. In spite of this, complete list of minimal $(s,k)$-polar
obstructions are known only for the cases $\min\{s,k\}=0$ and $s = k = 1$. A
graph is $(0,1)$-polar if and only if it is complete, so the only minimal
$(0,1)$-polar obstruction is $\overline{K_2}$. For an integer $k$, $k \ge 2$,
the $(0,k)$-polar graphs are called a \textit{$k$-clusters}, and they coincide
with the $\{P_3, (k+1)K_1\}$-free graphs. The $(0,\infty)$-polar graphs are
simply called \textit{clusters}, and they are the $P_3$-free graphs. Analogous
forbidden subgraph characterizations can be given for the case $k=0$ by
considering that the complement of an $(s,k)$-polar graph is a $(k,s)$-polar
graph.

The $(1,1)$-polar graphs are usually called \textit{split graphs}. The following
marvelous characterizations of split graphs were provided by Foldes, Hammer, and
Simeone.

\begin{theorem}[\cite{foldesSECGTC,hammerC1}] \label{theo:recognitionSplit}
    Let $G$ be a graph with vertex set $\set{v_1 ,v_2, \dots, v_n}$ and degree
    sequence $d_1 \ge d_2 \ge \dots \ge d_n$, where $d_i$ is the degree of
    vertex $v_i$. Set $p = \max\set{i : d_i \ge i-1}$. The following conditions
    are equivalent:
    \begin{enumerate}
        \item $G$ is a split graph;

        \item $G$ is a $\set{2K_2, C_4, C_5}$-free graph;

        \item \( \sum_{i=1}^p d_i= p(p-1) + \sum_{i=p+1}^n d_i. \)
    \end{enumerate}
    Additionally, if $G$ is a split graph, then $(\set{v_1, \dots, v_p},
    \set{v_{p+1}, \dots, v_n})$ is a split partition of $G$, $\omega(G) =
    \chi(G) = p$, and $\alpha(G) = \theta(G) = n-\min\set{p,d_p}$.
\end{theorem}

Note that, once the degree sequence of a graph $G$ is known, computing the value
of $p$, as well as verifying the condition in item 3, can be done in
$O(|V|)$-time, so split graphs are recognizable and a split partition can be
found in linear time from their degree sequences.

Maffray and Preissmann \cite{maffrayDAM52} introduced the following
generalization of split graphs. Given a fixed graph $H$, a graph $G$ is said to
be \textit{$H$-split} if $V_G$ admits a partition $(C, S, I)$ such that $C$ is a
clique, $I$ is an independent set, either $S = \varnothing$ or $G[S] \cong H$,
$C$ is completely adjacent to $S$, and $I$ is completely nonadjacent to $S$. A
partition $(C, S, I)$ as described above is called an \textit{$H$-split
partition} of $G$. If $G = (C, S, I)$ is an $H$-split graph with $S \ne
\varnothing$, we say that $G$ is an \textit{strict $H$-split graph}. Given a
family of graphs $\mathcal{H}$, we say that $G$ is \textit{$\mathcal{H}$-split}
if it is $H$-split for some $H \in \mathcal{H}$. The next theorem implies that,
for a graph $H$ whose degree sequence is uniquely realizable, the class of
$H$-split graphs is recognizable and an $H$-split partition can be found in
$O(|V|)$-time from their degree sequences.

\begin{theorem}[\cite{maffrayDAM52}] \label{theo:charHsplit}%
    Let $d_1^\ast \ge \dots \ge d_h^\ast$ be a realizable degree sequence and
    let $\mathcal{H}$ be the class of all realizations of this sequence. Let $G$
    be a graph with degree sequence $d_1 \ge \dots \ge d_n$. If $q = \max \set{i
    : d_i \ge i-1 + h } \cup \set{0}$, then $G$ is an $\mathcal{H}$-split graph
    if and only if $G$ is split or
    \[ \sum_{i=1}^q d_i = q(q-1) +qh + \sum_{j=q+h+1}^n d_i \]
    and $d_{q+i} = q + d_i^\ast$ for each $i \in \set{1, \dots, h}$.
    Additionally, if the condition on the degrees holds, then the sets $C =
    \set{v_1, \dots, v_q}$, $S = \set{v_{q+1}, \dots, v_{q+h}}$ and $I =
    \set{v_{q+h+1}, \dots, v_n}$ conform an $\mathcal{H}$-partition of $G$ , the
    subgraph induced by $S$ being isomorphic to some graph $H \in \mathcal{H}$.
    Moreover, if $d^\ast$ is a uniquely realizable degree sequence, then
    $\omega(G) = q + \omega(H)$, $\chi(G) = q+\chi(H)$, $\alpha(G) = \alpha(H) +
    n-q-h$ and $\theta(G) = \theta(H) + n-q-h$.
\end{theorem}

Notice that $H$-split graphs conform a hereditary class of graphs if and only if
either $H$ is a split graph, in which case $H$-split graphs coincide with split
graphs, or $H$ is one of the three minimal split obstructions mentioned in
\Cref{theo:recognitionSplit}, i.e., if $H \in \set{2K_2, C_4, C_5}$. The
following observation will be frequently used in this work without any explicit
mention.

\begin{remark}
    Let $H$ be some of $2K_2, C_4$, or $C_5$, and let $G = (C, S, I)$ be a
    strict $H$-split graph. Then, the only induced copy of $H$ in $G$ is $G[S]$
    and the $H$-split partition of $G$ is unique.
\end{remark}

The class of $H$-split graphs is self-complementary if and only if $H$ is. From
the above observations, it is not strange that the most studied $H$-split graphs
are the $C_5$-split graphs which were named \textit{pseudo-split graphs} in
\cite{maffrayDAM52}. Naturally, a $C_5$-split partition of a graph is also
called a \textit{pseudo-split partition} and, since pseudo-split graphs are
strict if and only if they are perfect, strict pseudo-split graphs are called
\textit{imperfect}. Additionally to the characterization of pseudo-split graphs
by their degree sequences provided by \Cref{theo:charHsplit}, Maffray and
Preissmann \cite{maffrayDAM52} gave the complete list of minimal pseudo-split
obstructions, namely $\set{2K_2, C_4}$. The following proposition summarize such
characterizations to facilitate future references.

\begin{theorem}[\cite{maffrayDAM52}] \label{theo:recognitionIPS}
    Let $G$ be a graph of order at least five with vertex set $\set{v_1, v_2,
    \dots, v_n}$ and degree sequence $d_1 \ge d_2 \ge \dots \ge d_n$, where
    $d_i$ is the degree of vertex $v_i$. Set $q = \max \set{i : d_i \ge i+4}
    \cup \set{0}$. The following conditions are equivalent:
    \begin{enumerate}
        \item $G$ is an imperfect pseudo-split graph;

        \item $G$ is a $\set{2K_2, C_4}$-free graph that has an induced $C_5$;

        \item \( \sum_{i=1}^q d_i= q(q+4) + \sum_{i=q+6}^n d_i\), and
        $d_j = q+2$ whenever $q+1 \le j \le q+5$.
    \end{enumerate}
    Additionally, if $G$ is an imperfect pseudo-split graph, then
    \[ (\set{v_1, v_2, \dots, v_q}, \set{v_{q+1},v_{q+2},v_{q+3},v_{q+4},v_{q+5}},
    \set{v_{q+6}, v_{q+7}, \dots, v_n}) \]
    is the pseudo-split partition of $G$, $\omega(G) = q+2$, $\chi(G) =
    q+3$, $\alpha(G) = n-q-3$ and $\theta(G) = n-q-2$.
\end{theorem}

In this work, we study $(s,k)$-polar partitions on pseudo-, $2K_2$- and
$C_4$-split graphs. Polarity on pseudo-split graphs is treated in \Cref{sec:PS},
where we provide finite lists of minimal obstructions for the main polar
properties, and give linear-time algorithms to recognize such properties on
pseudo-split graphs from their degree sequences; at the end of the section we
study $(k, \ell)$-colorings of pseudo-split graphs. Results about polarity in
$2K_2$-split graphs that are analogous to those given in \Cref{sec:PS} for
pseudo-split graphs are developed in \Cref{chap:2K2split}. We show that, among
other differences, $2K_2$-split graphs that are $(s,k)$-polar cannot be
recognized from their degree sequence as pseudo-split $(s,k)$-polar graphs do,
but they are still efficiently recognizable. Since $C_4$-split graphs are the
complements of $2K_2$-split graphs, analogous results are deduced for these
graphs. Finally, in \Cref{sec:concl} we pose some open problems and conjectures.

% % % % % % % % % % % % % % % % % % % % % % % % % % % % % % % % % % % % % % % %
\section{Polarity on pseudo-split graphs} \label{sec:PS}
% % % % % % % % % % % % % % % % % % % % % % % % % % % % % % % % % % % % % % % %

As we have observed before, split graphs are precisely the $(1,1)$-polar graphs,
so they trivially are polar, monopolar, unipolar, and $(s, k)$-polar for any
positive integers $s$ and $k$. In this section we study polarity on pseudo-split
graphs. As our main results we give complete lists of pseudo-split minimal $(s,
k)$-polar obstructions for the cases $\min \set{s, k} \le 2$, $s = \infty$, and
$k = \infty$, we prove tight upper bounds for the order of pseudo-split minimal
$(s, k)$-polar obstruction, and provide $O(|V|)$-time algorithms to decide
whether a pseudo-split graph is $(s, k)$-polar from its degree sequence.

The next observation is basic to obtain $O(|V|)$-time recognition algorithms for
$(s, k)$-polarity on pseudo-split graphs; it follows directly from
\Cref{theo:recognitionIPS}.

\begin{remark} \label{rem:degrees}
    Let $G = (C, S, I)$ be an imperfect pseudo-split graph, and let $c=|C|$ and
    $i=|I|$. If $u \in V_G$, then $u \in C$ if and only if $\Deg{u} \ge c+4$, $u
    \in S$ if and only if $\Deg{u} = c+2$, and $u \in I$ if and only if $\Deg{u}
    \le c$. Moreover, $\Deg{u} = c+4$ if and only if $u$ is a vertex in $C$ that
    is completely nonadjacent to $I$, and $\Deg{u} = c$ if and only if $u$ is a
    vertex in $I$ that is completely adjacent to $C$. 
\end{remark}

If $G = (C, S, I)$ is a pseudo-split graph and $uv$ is an edge in $G[S]$, then
$(C \cup S \setminus \set{u,v}, I \cup \set{u,v})$ is a polar partition of $G$.
Hence, pseudo-split graphs are polar. In addition, since split graphs are
precisely the $(1,1)$-polar graphs, for every pair of positive integers $s$ and $k$,
any pseudo-split minimal $(s, k)$-polar obstruction is necessarily imperfect.
Moreover, since split graphs are monopolar and unipolar, every pseudo-split
minimal monopolar (unipolar) obstruction is also imperfect. Considering above
observations, it seems natural to ask about the polar partitions of $C_5$.
Notice that a $5$-cycle admits only two essentially different polar partitions,
which are depicted in \Cref{fig:polarPartitionsOfC5}.

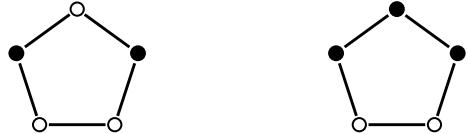
\begin{figure}[!ht]
\centering
\begin{tikzpicture}

    \begin{scope}[scale=0.85]
    
        \begin{scope}[scale=1]
            \foreach \i in {0,2,3}
                \node [vertex] (\i) at ($ (0,2) + ({90+(\i*72)}:1) $)[]{};

            \foreach \i in {1,4}
                \node [blackV] (\i) at ($ (0,2) + ({90+(\i*72)}:1) $)[]{};

            \foreach \i in {0,...,4}
                \draw let \n1={int(mod(\i+1,5))} in [edge] (\i) to node [above] {} (\n1);
        \end{scope}
        
        \begin{scope}[scale=1, xshift=5cm]
            \foreach \i in {2,3}
                \node [vertex] (\i) at ($ (0,2) + ({90+(\i*72)}:1) $)[]{};

            \foreach \i in {0,1,4}
                \node [blackV] (\i) at ($ (0,2) + ({90+(\i*72)}:1) $)[]{};

            \foreach \i in {0,...,4}
                \draw let \n1={int(mod(\i+1,5))} in [edge] (\i) to node [above] {} (\n1);
        \end{scope}
        
    \end{scope}
    
\end{tikzpicture}

\caption{The only two polar partitions of a $5$-cycle. Shaded vertices
induce complete multipartite graphs, while white vertices induce clusters.}
\label{fig:polarPartitionsOfC5}
\end{figure}

Observe that, if an imperfect pseudo-split graph $G = (C, S, I)$ has a polar
partition $(A, B)$, then $(A, B)$ must inherit either a $(1,2)$- or a
$(2,1)$-polar partition to $G[S]$. In the first case, since $G[B]$ is a
$P_3$-free graph and $C$ is completely adjacent to $S$, we have that $C \cap B$
must be an empty set, so $C \subseteq A$. Analogously, when $G[S]$ inherits a
$(2,1)$-polar partition from $(A, B)$, we have that $I \subseteq B$, because $I$
is completely nonadjacent to $S$ and $A$ induces a $\overline{P_3}$-free graph.
These observations are going to be used without any explicit mention in most of
the proofs of this section.

% % % % % % % % % % % % % % % % % % % % % % % % % % % % % % % % % % % % % % % %
\subsection{Algorithms for polarity on pseudo-split graphs}
% % % % % % % % % % % % % % % % % % % % % % % % % % % % % % % % % % % % % % % %

In the following theorem we give a necessary and sufficient condition for a
pseudo-split graph to be $(s, \infty)$-polar. Notice that such a condition can
be verified in $O(|V|)$-time from the degree sequence of a graph. Additionally,
since the class of pseudo-split graphs is self-complementary, and a graph is
$(s, \infty)$-polar if and only if its complement is $(\infty, s)$-polar, we
have that by a simple argument of complements an analogous characterization can
be given for $(\infty, k)$-polarity on pseudo-split graphs.

\begin{theorem} \label{thm:psSinfty}
    Let $s$ be a nonnegative integer, and let $G$ be an imperfect pseudo-split
    graph with pseudo-split partition $(C, S, I)$. Then, $G$ is an $(s,
    \infty)$-polar graph if and only either $s > |C|$, or $|C| \ge s \ge 2$ and
    there are at least $|C|-s+2$ vertices of $G$ with degree exactly $|C|+4$.
\end{theorem}

\begin{proof}
    Let us denote $|C|$ by $c$. Suppose that $G$ is an $(s, \infty)$-polar
    graph, with polar partition $(A, B)$, such that $s \le c$. Observe that, if
    the restriction of $(A, B)$ to $S$ is a $(1,2)$-polar partition, then $C
    \cap B = \varnothing$, or $G[B]$ has $P_3$ as an induced subgraph, but then
    $C \subseteq A$, which is impossible since $G[A]$ would have $K_{c+1}$ as an
    induced subgraph and $s \le c$. Thus, $G[S]$ is covered by a $(2,1)$-polar
    partition, so $s \ge 2$. Notice that $I \subseteq B$, otherwise
    $\overline{P_3}$ would be an induced subgraph of $G[A]$, which cannot occur.

    Then, since $A$ induces a complete $s$-partite graph, at most $s-2$ vertices
    of $C$ belong to $A$. It implies that there is a subset $C'$ of $C \cap B$
    with at least $c-s+2$ vertices. Moreover, if there exist adjacent vertices
    $c \in C'$ and $i \in I$, then $G[B]$ would have $P_3$ as an induced
    subgraph, which is impossible, so $C'$ is completely nonadjacent to $I$.
    Hence, by \Cref{rem:degrees}, $G$ has at least $c-s+2$ vertices of degree
    $c+4$.
    
    For the converse implication, let $A$ and $B$ be a maximum independent set
    and a maximum clique in $G[S]$, respectively. If $s > c$, then $(C \cup A, I
    \cup S \setminus A)$ is an $(s, \infty)$-polar partition of $G$. Otherwise,
    we have that $c \ge s$ and there are at least $c-s+2$ vertices of $G$ with
    degree exactly $c+4$. But then, if $C'$ is the subset of $C$ consisting of
    the vertices of degree $c+4$, we have by \Cref{rem:degrees} that $((S
    \setminus B) \cup (C \setminus C'), I \cup C' \cup S \setminus B)$ is an
    $(s, \infty)$-polar partition of $G$.
\end{proof}

Now, we present a necessary and sufficient condition for a pseudo-split graph to
be $(s, k)$-polar. Once again, this condition can be verified in $O(|V|)$-time
from the degree sequence of a graph, so it implies that $(s, k)$-polarity can be
efficiently decided on pseudo-split graphs.

\begin{theorem} \label{theo:DegreeChar(sk)polIPS}
    Let $G$ be an imperfect pseudo-split graph with pseudo-split partition $(C,
    S, I)$, and let $c$ and $i$ be the cardinalities of $C$ and $I$,
    respectively. Let $M_C$ be the number of vertices of $G$ whose degree is
    exactly $c+4$, and $M_I$ be the number of vertices of $G$ whose degree is
    exactly $c$. Let $s$ and $k$ be nonnegative integers such that $s+k \ge 1$.
    Then, $G$ is an $(s, k)$-polar graph if and only if either
    \begin{enumerate}
        \item $k \ge i+1$ and $s \ge c-M_C+2$, or
        \item $s \ge c+1$ and $k \ge i-M_I+2$.
    \end{enumerate}
\end{theorem}

\begin{proof}
    First, let us suppose that $G$ admits an $(s, k)$-polar partition $(A, B)$.
    There are two possible cases, either $G[S]$ inherits a $(2, 1)$-polar
    partition from $(A,B)$, or it inherits a $(1, 2)$-polar partition (see
    \Cref{fig:polarPartitionsOfC5}). We will show that in the first case, $k \ge
    i+1$ and $s \ge c-M_C-2$, while in the latter case $s \ge c+1$ and $k \ge
    i-M_I+2$.

    Thus, suppose that $G[S]$ inherits a $(2,1)$-polar partition from $(A, B)$.
    Note that in such a case $I \subseteq B$, otherwise $G[A]$ would have
    $\overline{P_3}$ as an induced subgraph, which is impossible. Moreover,
    since $G[B]$ is a $P_3$-free graph, we have that every vertex $v \in C \cap
    B$ is completely nonadjacent to $I$, and then, by \Cref{rem:degrees},
    $|C\cap B| \le M_C$. Thus, it occurs that
    \[ |C\cap A| = |C|-|C\cap B| \ge c-M_C, \]
    where we conclude that $s \ge c-M_C+2$. Furthermore, in this case $G[B]$ is
    a cluster with exactly $I + 1$ components, so $k \ge i+1$. Hence, we have
    proved that, if $G[S]$ inherits a $(2, 1)$-polar partition from $(A, B)$,
    then $k \ge i+1$ and $s \ge c-M_C+2$. It can be proved in a similar way that
    $s \ge c+1$ and $k \ge i-M_I+2$ whenever $G[S]$ inherits a $(1, 2)$-polar
    partition from $(A, B)$.

    Conversely, let assume that $k \ge i+1$ and $s\ge c-M_C+2$. By definition of
    $M_C$ and \Cref{rem:degrees}, there exists a subset $C'$ of $C$ of
    cardinality $M_C$ that is completely nonadjacent to $I$. Let $B_1$ be a set
    of two adjacent vertices of $G[S]$, and let $A_1 = S \setminus B_1$. Then,
    we have that $(A_1 \cup C \setminus C', B_1 \cup I \cup C')$ is a $(c-M_C+2,
    i+1)$-polar partition of $G$, so $G$ is an $(s, k)$-polar graph, as we had
    to prove. The result follows analogously if we assume that $s \ge c+1$ and
    $k \ge i-M_I+2$, only taking a $(1, 2)$-polar partition $(A_1, B_1)$ of
    $G[S]$ instead of a $(2, 1)$-polar partition.
\end{proof}

Monopolarity and unipolarity also can be decided in $O(|V|)$-time from the
degree sequence of a pseudo-split graph, but that will be deduced as an
immediate consequence of the forbidden subgraph characterizations that we
present next.

Now, we present some results about pseudo-split minimal $(s, k)$-polar
obstructions, which include tight upper bounds for the order of such graphs, as
well as complete lists of minimal forbidden induced subgraphs for some values of
$s$ and $k$.

% % % % % % % % % % % % % % % % % % % % % % % % % % % % % % % % % % % % % % % %
\subsection{Pseudo-split minimal $(s, k)$-polar obstructions}
% % % % % % % % % % % % % % % % % % % % % % % % % % % % % % % % % % % % % % % %

Monopolar an unipolar pseudo-split graphs admit a very simple characterization
by forbidden induced subgraphs which we summarize in the following proposition.

\begin{theorem} \label{theo:charIPSobminMonopUnip}
    Let $G$ be a pseudo-split graph. Then,
    \begin{enumerate}
        \item $G$ is a minimal monopolar obstruction if and only if $G \cong K_1
        \oplus C_5$, and

        \item $G$ is a minimal unipolar obstruction if and only if $G \cong
        C_5$.
    \end{enumerate}
    In consequence, the problems of deciding whether a pseudo-split graph is
    monopolar or unipolar are solvable in $O(|V|)$-time from its degree
    sequence.
\end{theorem}

\begin{proof}
    It is a routine to show that $K_1 \oplus C_5$ is a minimal monopolar
    obstruction. Moreover, if $G$ has pseudo-split partition $(C, S, I)$, and
    $G$ does not have $K_1 \oplus C_5$ as an induced subgraph, then either $G$
    is a split graph, or $G$ is an imperfect pseudo-split graph with $C =
    \varnothing$. In the first case, $G$ trivially is a monopolar graph, while
    in the second case $G$ is isomorphic to $nK_1+C_5$ for some nonnegative
    integer $n$, and therefore, it is a monopolar graph. The proof of item 2 is
    similar and even simpler.

    By item 1, a pseudo-split graph $G = (C, S, I)$ is monopolar if and only if
    either $S = \varnothing$ or $C = \varnothing$. Thus, it follows from
    \Cref{theo:recognitionSplit,theo:recognitionIPS} that deciding whether a
    pseudo-split graph is monopolar can be done in $O(|V|)$-time from its degree
    sequence. Analogously, by item 2, a pseudo-split graph is unipolar if and
    only if it is split, so in this case the result follows from
    \Cref{theo:recognitionSplit}.
\end{proof}

As we mentioned before, minimal $(s, k)$-polar obstructions on general graphs
are known only for the cases $\min\set{s, k} = 0$, and $s = k =1$, which
correspond to clusters, complete multipartite graphs, and split graphs.  In the
following proposition we give complete lists of pseudo-split minimal $(s,
k)$-polar obstructions for the case $s \in \set{1, 2}$, which can be
extrapolated to case $k \in \set{1,2}$ by simple arguments of complements.
Before presenting such results, we introduce notation for some particular
graphs.

For each positive integer $s$, let us denote by $G_s^0$ the imperfect
pseudo-split graph whose $(C, S, I)$-partition satisfies that $|C| = s$, $I =
1$, and $C$ is completely adjacent to $I$. We will also use $G_s^1$ to denote
the graph obtained from $G_s^0$ by deleting one edge incident with the only
vertex of $I$. Notice that, by \Cref{theo:DegreeChar(sk)polIPS}, for any
integers $s, k \ge 2$, the graphs $G_s^0$ and $G_s^1$ are minimal $(s, k)$-polar
obstructions.

For positive integers $s$ and $k$, with $k \ge s$, let $H_s^k = (C, S, I)$ be
the imperfect pseudo-split graph such that $|C| = s-1$, $|I| = k-1$ and, for an
injection $f \colon C \to I$, a vertex $v \in C$ is adjacent to a vertex $u \in
I$ if and only if $u = f(v)$. It follows from \Cref{theo:DegreeChar(sk)polIPS}
that $H_s^k$ is a minimal $(s, k)$-polar obstruction provided $k \ge s \ge 2$.

\begin{theorem} \label{theo:char1kand2kIPSobsmin}
    Let $k$ be an integer, $k \ge 2$. If $G$ is a pseudo-split graph, then
    \begin{enumerate}
        \item $G$ is a minimal $(1, k)$-polar obstruction if and only if $G
        \cong K_1 \oplus C_5$.

        \item $G$ is a minimal $(2, k)$-polar obstruction if and only if $G$ is
        isomorphic to some of $G_2^0 ,G_2^1, \overline{G_k^0}$ or $H_1^k$.
    \end{enumerate}
\end{theorem}

\begin{proof}
    It is a routine to verify that $K_1 \oplus C_5$ is a minimal $(1, k)$-polar
    obstruction. In addition, $G$ is $K_1 \oplus C_5$-free if and only if $S =
    \varnothing$ or $C = \varnothing$, but in both cases $G$ is a $(1,2)$-polar
    graph, hence a $(1, k)$-polar graph.

    Previously, we observed that the graphs $G_2^0 ,G_2^1$ and $H_2^k$ are all
    of them $(2, k)$-polar obstructions. We also observed that $G_k^0$ is a
    minimal $(k, 2)$-polar obstruction, so $\overline{G_k^0}$ is a minimal $(2,
    k)$-polar obstruction.
    
    Now, assume for having a contradiction that $G$ is a minimal $(s, k)$-polar
    obstruction different to $G_2^0 ,G_2^1, \overline{G_k^0}$ and $H_1^k$. Let
    $(C, S, I)$ be the pseudo-split partition of $G$, and let us denote by $c$
    and $i$ the cardinalities of $C$ and $I$, respectively. Notice that $G$ is
    imperfect, otherwise it would be a $(1,1)$-polar graph, and hence a
    $(2,k)$-polar graph. Also observe that, if either $i = 0$, or both $c \le 1$
    and $i \le k-1$, then $G$ would admit a $(2, k)$-polar partition, which is
    impossible. From the previous observation we have that either $c \ge 2$ and
    $i \ge 1$, or $c \le 1$ and $i \ge k$.

    Suppose that $c \ge 2$ and $i \ge 1$. Since $G$ is a $\set{G_2^0 ,G_2^1}$-free
    graph, we have that $C$ is completely nonadjacent to $I$. Notice that $i \ge
    k$, otherwise $G$ would be a $(2, k)$-polar graph. But then, $G$ has
    $\overline{G_k^0}$ as an induced subgraph, which is impossible. Thus, it
    must be the case that $c \le 1$ and $i \ge k$. Since $G$ is not a $(2,
    k)$-polar graph, $c \ge 1$, so $c = 1$. Let $v$ be the only vertex in $C$.
    If $|N(v) \cap I| \le k-2$, then $G$ is a $(2, k)$-polar graph, which is not
    possible, so that $|N(v) \cap I| \ge k-1$, but then $G$ contains an induced
    subgraph isomorphic to either $\overline{G_k^0}$ or $H_1^k$, contradicting
    that $G$ is not isomorphic to these graphs. The contradiction arose from
    supposing the existence of a pseudo-split minimal $(2,k)$-polar obstruction
    different to $G_2^0 ,G_2^1, \overline{G_k^0}$ and $H_1^k$, so it does not
    exist.
\end{proof}

It is a simple observation that, for any nonnegative integer $s$, a graph $G$ is
a minimal $(s, \infty)$-polar obstruction if and only if there is a nonnegative
integer $k_0$ such that, for any integer $k \ge k_0$, $G$ is a minimal $(s,
k)$-polar obstruction.
Since monopolar graphs are by definition the $(1, \infty)$-polar graphs, then
item 1 of \Cref{theo:charIPSobminMonopUnip} can be deduced as a consequence of
the previous observation and item 1 of \Cref{theo:char1kand2kIPSobsmin}.
Similarly, the next corollary follows directly from the previous observation and
item 2 of \Cref{theo:char1kand2kIPSobsmin}.

\begin{corollary} \label{cor:ps2inftyObsmin}
    There are only two pseudo-split minimal $(2, \infty)$-polar obstructions,
    namely $G_2^0$ and $G_2^1$.
\end{corollary}

It seems that there is not an easy way to describe the complete lists of
pseudo-split minimal $(s, k)$-polar obstructions when $s$ and $k$ are arbitrary
nonnegative integers, but, as we observed before, we known there is just a
finite number of them, so it becomes natural to ask about the best possible
upper bounds for their order. In the following propositions, we use
\Cref{theo:DegreeChar(sk)polIPS} to give tight upper bounds for the order of
pseudo-split minimal $(s, k)$- and $(s, \infty)$-polar obstructions. We start
with the following technical observation.

\begin{lemma} \label{lem:SomeNoMin(sk)obs}
    Let $G = (C, S, I)$ be an imperfect pseudo-split graph, and let $c$ and $i$
    be the cardinalities of $C$ and $I$, respectively. Let $s$ and $k$ be
    nonnegative integers such that $s+k \ge 1$. The following assertions hold
    true.
    \begin{enumerate}
        \item If $c > s$, then $G$ is not a minimal $(s, k)$-polar obstruction.

        \item If $i > k$, then $G$ is not a minimal $(s, k)$-polar obstruction.
    \end{enumerate}
\end{lemma}

\begin{proof}
    We only prove item 1 because the proof of item 2 is analogous. Notice that,
    if $c > s$ and $i \ge k$, it follows from \Cref{theo:DegreeChar(sk)polIPS}
    that, for every vertex $v \in C$, $G-v$ is not an $(s, k)$-polar graph, which
    clearly implies that $G$ is not a minimal $(s, k)$-polar obstruction. Thus,
    we can assume that $i < k$.
    
    Assume for obtaining a contradiction that $G$ is a minimal $(s, k)$-polar
    obstruction. Observe that for every vertex $v$ of $C$, $G-v$ has
    pseudo-split partition $(C\setminus\set{v},S,I)$, and $|C\setminus\set{v}| =
    c-1 \ge s$. Since $G$ is a minimal $(s, k)$-polar obstruction, we have that,
    for every vertex $v$ of $G$, $G-v$ admits an $(s, k)$-polar partition. That
    is true in particular if $v\in C$. Then, we have from \Cref{rem:degrees},
    \Cref{theo:DegreeChar(sk)polIPS}, and our previous observations that, for
    any vertex $v$ of $C$, there are at least $c-s+1$ vertices of $C \setminus
    \set{v}$ that are completely nonadjacent to $I$; let $C'_v$ be the set of
    these vertices.

    Notice that no vertex $v$ of $C$ is completely nonadjacent to $I$, otherwise
    $C'_v \cup \set{v}$ would be a subset of $C$ of cardinality at least $c-s+2$
    that is completely nonadjacent to $I$, but then, by
    \Cref{rem:degrees,theo:DegreeChar(sk)polIPS}, $G$ would be an $(s, k)$-polar
    graph, and we are assuming it is not. Thus, we conclude that each vertex of
    $C$ is adjacent to at least one vertex of $I$.

    Here is the desired contradiction. Let $H$ be a graph obtained from $G$ by
    removing any $c-s$ vertices of $C$. Thus, $H$ is a proper induced subgraph
    of $G$ with a pseudo-split partition $(C^\ast, S,I)$ such that any vertex of
    $C^\ast$ is adjacent to at least one vertex of $I$. But then, we have from
    \Cref{theo:DegreeChar(sk)polIPS} that $H$ is not an $(s, k)$-polar graph,
    contradicting the minimality of $G$.
\end{proof}

By itself, \Cref{lem:SomeNoMin(sk)obs} implies that, for any nonnegative
integers $s$ and $k$, a pseudo-split minimal $(s, k)$-polar obstruction has
order at most $s + k + 5$. Nevertheless, as we can corroborate in
\Cref{theo:char1kand2kIPSobsmin} and the observations that precede it, if $\min
\set{s, k} \le 2$, each minimal $(s, k)$-polar obstruction has order strictly
lower than $s + k +5$. In
\Cref{lem:fin(sk)obsminIPSmin2,theo:fin(sk)obsminIPSmin3} we will prove that
this is true for general values of $s$ and $k$, and not only when $\min \set{s,
k} \le 2$.

\begin{lemma} \label{lem:fin(sk)obsminIPSmin2}
    Let $s$ and $k$ be integers, $s, k \ge 2$, and let $G = (C, S, I)$ be a
    pseudo-split minimal $(s, k)$-polar obstruction. Then, $G$ is imperfect,
    $|C| \le s$, $|I| \le k$ and $|C|+|I| \le s+k -1$.
    
    Consequently, any pseudo-split minimal $(s, k)$-polar obstruction has order
    at most $s + k + 4$, and this bound is tight when $\min \set{s, k} = 2$.
\end{lemma}

\begin{proof}
    As we noticed at the beginning of this chapter, split graphs are $(1,1)$-polar,
    hence $(s, k)$-polar, so $G$ is an imperfect pseudo-split graph. Let $c =
    |C|$ and $i = |I|$. Observe that \Cref{lem:SomeNoMin(sk)obs} implies that $s
    \ge c$ and $k \ge i$, so $|V_G| = |C| + |I| + |S| \le s + k + 5$ and this
    bound is attained if and only if $c = s$ and $i = k$.

    Assume for a contradiction that $c = s$ and $i = k$, so $G$ has order $s +k
    +5$. Let $v \in C$, and let us use $C'$ to denote $C \setminus \set{v}$ . Let
    $(A, B)$ be an $(s, k)$-polar partition of $G-v$. Observe that $G[S]$
    inherit a $(1, 2)$-polar partition from $(A, B)$, otherwise $G[S]$ would
    inherit a $(2, 1)$-polar partition, but then $I \subseteq B$ implying that
    $B$ has an independent subset of size $k+1$, which is impossible. Moreover,
    since $G[B]$ is $\set{(k+1)K_1, P_3}$-free, we have that $C' \subseteq A$ and
    there is at least one vertex $u$ of $I$ in the part $A$. Notice that $u$ is
    completely adjacent to $C'$, because $G[A]$ does not have induced copies of
    $\overline{P_3}$.

    Here we have the desired contradiction, because $G[C \cup S \cup \set{u}]$ is
    isomorphic to either $G_s^0$ or $G_s^1$, depending on whether $u$ is
    adjacent or not to $v$, but then $G$ has an $(s, k)$-polar obstruction as a
    proper induced subgraph, contradicting the minimality of $G$. The
    contradiction arose from assuming that $|V_G| > s+k+4$, so it is not the
    case. Notice that $G_s^0$ and $G_s^1$ attain the bound when $s = 2$, so the
    bound is tight.
\end{proof}

\begin{theorem} \label{theo:fin(sk)obsminIPSmin3}
    Let $s$ and $k$ be integers, $s, k \ge 3$. Then, any pseudo-split minimal
    $(s, k)$-polar obstruction has order at most $s + k + 3$, and the bound is
    tight.
\end{theorem}

\begin{proof}
    Let $G$ be an $(s, k)$-polar obstruction with pseudo-split partition $(C, S,
    I)$, and let $c$ and $i$ be the cardinalities of $C$ and $I$, respectively.
    By \Cref{lem:fin(sk)obsminIPSmin2}, $G$ is an imperfect pseudo-split graph
    with $c \le s$, $i \le k$ and, either $c \le s-1$ or $i \le k-1$. Notice
    that, if $c < s-1$ or $i < k-1$, then $|V_G| \le s + k -3$, so we are done.
    Thus we can assume that, either $c = s-1$ or $i = k-1$. Let us assume that
    $i = k-1$, the case $c = s-1$ is analogous.

    To obtain a contradiction, let us suppose that $G$ has at least $s + k + 4$
    vertices, which implies by the previous observations that $c = s$. Let $v$
    be a vertex in $C$, and let $(A, B)$ be an $(s, k)$-polar partition of
    $G-v$. We have two cases: either $G[S]$ inherits a $(1,2)$- or a
    $(2,1)$-polar partition from $(A, B)$.

    In the first case, since $G[B]$ is $\set{(k+1)K_1, P_3}$-free, we have that $C
    \setminus \set{v} \subseteq A$ and there exists a vertex $u \in I\cap A$.
    Moreover, $G[A]$ is a $\overline{P_3}$-free graph, so $u$ is completely
    adjacent to $C \setminus \set{v}$. But then, $G[C \cup S \cup \set{u}]$ is
    isomorphic to either $G_s^0$ or $G_s^1$, so $G$ properly contains an
    $(s, k)$-polar obstruction, which is impossible.
    
    Hence, it must be the case that $G[S]$ inherits a $(2,1)$-polar partition
    $(A', B')$ from $(A, B)$, in which case $I \subseteq B$ and there exists a
    vertex $u \in B \cap (C \setminus \set{v})$, so that $u$ is completely non
    adjacent to $I$. Additionally, repeating the argument, but using $u$ instead
    of $v$, we have that there exists a vertex $w \in B \cap (C \setminus
    \set{u})$, so that $w$ is completely non adjacent to $I$. But then,
    \[ (A' \cup C \setminus \set{u, w}, B' \cup I \cup \set{u,w}) \] is an $(s,
    k)$-polar partition of $G$, a contradiction. The contradiction arose from
    supposing that $|V_G| \ge s+ k + 4$, so it must be the case that $G$ has at
    most $s + k + 3$ vertices.

    To bound is tight since $H_s^k$ is a pseudo-split minimal $(s, k)$-polar
    obstruction whenever $k \ge s \ge 3$, and $\overline{H_k^s}$ is a
    pseudo-split minimal $(s, k)$-polar obstruction provided $s \ge k \ge 3$.
\end{proof}

In contrast with minimal $(s, k)$-polar obstructions when $s$ and $k$ are
integers, it is unknown whether the number of minimal $(s, \infty)$-polar
obstructions is finite. In the following propositions we prove that, restricted
to the class of pseudo-split graphs, minimal $(s, \infty)$-polar obstructions
are, all of them, minimal $(s, s+1)$-minimal obstructions, implying that there
is only a finite number of them. We start with some technical propositions.

\begin{lemma} \label{lem:ChracterizationCvC'v}
    Let $s$ and $k$ be positive integers, and let $G = (C, S, I)$ be an
    imperfect pseudo-split graph such that $|C| = s$ and $0 < |I| < k-1$. For
    each vertex $v \in I$, let $C_v = \set{ w \in C : w \notin N(v) }$ and
    let $C'_v$ be the set of all vertices in $C$ that are completely
    nonadjacent to $I \setminus \set{v}$. Then, $G$ is a minimal $(s, k)$-polar
    obstruction if and only if for each $v \in I$, both $|C'_v| \ge 2$ and $|C_v
    \cap C'_v| \le 1$.
\end{lemma}

\begin{proof}
    Suppose that $G$ is a minimal $(s, k)$-polar obstruction. From the minimality
    of $G$ we have that, for each vertex $v \in I$, $G-v$ is an $(s, k)$-polar
    graph. Then, by \Cref{theo:DegreeChar(sk)polIPS}, $|C'_v| \ge 2$. Moreover,
    from the same proposition we have that, if $|C_v \cap C'_v| \ge 2$ for some
    $v \in I$, then $G$ is an $(s, k)$-polar graph, which is impossible. Then it
    must be the case that, for each vertex $v \in I$, $|C_v \cap C'_v| \le 1$.

    For the converse, assume that $|C'_v|\ge 2$ and $|C_v \cap C'_v| \le 1$, for
    each $v\in I$. For any vertex $v \in C$, $G-v$ is a pseudo-split graph whose
    complete part has $s-1$ vertices and whose independent part has at most
    $k-2$ vertices, so it follows from \Cref{theo:DegreeChar(sk)polIPS} that
    $G-v$ admits an $(s, k)$-polar partition. For any vertex $v \in I$, the set
    $C'_v$ has at least two vertices, so it also follows from
    \Cref{theo:DegreeChar(sk)polIPS} that $G-v$ is an $(s, k)$-polar graph. For
    any vertex $v \in S$, $G-v$ is a split graph, so $G-v$ is an $(s, k)$-polar
    graph. In summary, for every vertex $v$ of $G$, $G-v$ is an $(s, k)$-polar
    graph. Furthermore, \Cref{theo:DegreeChar(sk)polIPS} implies that $G$ is an
    $(s, k)$-polar graph if and only if there are at least two vertices of $C$
    that are completely nonadjacent to $I$. Nevertheless, if $C'$ is any subset
    of $C$ that is completely nonadjacent to $I$, then $C' \subseteq C_v \cap
    C'_v$ for any vertex $v \in I$, and therefore $|C'| \le 1$. Hence, $G$ is
    not an $(s, k)$-polar graph, and we conclude that $G$ is a minimal $(s,
    k)$-polar obstruction.
\end{proof}

For each integer $s \ge 2$, let $F_s = (C, S, I)$ be the imperfect pseudo-split
graph such that $|C| = s$, $|I| = s-1$ and, for an injection $f \colon I \to C$,
a vertex $v \in I$ is adjacent to a vertex $u \in C$ if and only if $u = f(v)$.
Notice that, from \Cref{theo:DegreeChar(sk)polIPS}, we have that for any
nonnegative integer $k$, $F_s$ is not an $(s, k)$-polar graph. Moreover, $F_s$
is a minimal $(s, k)$-polar obstruction if and only if $k > s$.

\begin{lemma} \label{lem:Iles-1}
    Let $s$ and $k$ be integers, $s, k \ge 3$, and let $G = (C, S, I)$ be a
    pseudo-split minimal $(s, k)$-polar obstruction such that $|C| = s$. Then,
    $|I| \le s-1$. In addition, if $|I| = s-1$, then $s < k$ and $G \cong F_s$.
\end{lemma}

\begin{proof}
    Let $c = |C|$ and $i = |I|$. For each $v \in I$, $G -v$ has an $(s,
    k)$-polar partition $(A, B)$. Moreover, since $ c = s$, $G[S]$ inherits a
    $(2,1)$-polar partition from $(A, B)$, so we have that $I \subseteq B$.
    Additionally, at least two vertices of $C$ belong to $B$, and any vertex in
    $C \cap B$ is completely nonadjacent to $I \setminus \set{v}$.

    Observe that, if two vertices in $C \cap B$ are nonadjacent to $v$, then $G$
    would have an $(s, k)$-polar partition, but this is not the case. Hence, for
    each vertex $v \in I$, there is a vertex $u \in C$ whose only neighbor in
    $I$ is $v$. Therefore, if $i \ge s$, $G$ properly contains the $(s,
    k)$-polar obstruction $F_s$, contradicting the minimality of $G$. Thus, we
    conclude that $i \le s-1$.

    Finally, if $i = s-1$, $G$ contains the $(s, k)$-polar obstruction $F_s$.
    But $G$ is a minimal $(s, k)$-polar obstruction, so $G \cong F_s$, and then
    $F_s$ is a minimal $(s, k)$-polar obstruction, which implies that $k > s$.
\end{proof}

\begin{lemma} \label{lem:sInftyIFFss+1}
    Let $s$ and $k$ be integers, $k > s \ge 3$, and let $G$ be a graph with
    pseudo-split partition $(C, S, I)$.
    \begin{enumerate}
        \item If $G$ is a minimal $(s, k)$-polar obstruction such that $|C| =
        s$, then $G$ is a minimal $(s, k')$-polar obstruction for each integer
        $k' \ge k$.

        Particularly, if $G$ is a minimal $(s, s+1)$-polar obstruction with $|C|
        = s$, then it is a minimal $(s, \infty)$-polar obstruction.

        \item If $G$ is a minimal $(s, k)$-polar obstruction such that $|C| =
        s$, then $G$ is a minimal $(s, s+1)$-polar obstruction.

        Consequently, if $G$ is a minimal $(s, \infty)$-polar obstruction, then
        it a minimal $(s, s+1)$-polar obstruction with $|C| = s$.
    \end{enumerate}

    In consequence, $G$ is a minimal $(s, \infty)$-polar obstruction if and only
    if $G$ is an $(s, s+1)$-polar obstruction with $|C| = s$.
\end{lemma}

\begin{proof}
    Let $k'$ be an integer, $k' \ge k$, and suppose that $G$ is a minimal $(s,
    k)$-polar obstruction such that $|C| = s$. Thus, we have from
    \Cref{lem:Iles-1} that, either $|I| < s-1$, or $k > s$ and $G \cong F_s$. In
    the latter case the result follows because $F_s$ is a minimal $(s,
    k')$-polar obstruction. Otherwise, we have that $|I| < s-1 < k-1$ and, by
    \Cref{lem:ChracterizationCvC'v} and the observation that precede it, $G$ is
    a minimal $(s, k')$-polar obstruction. The assertion follow easily from
    here.

    To prove item 2, assume again that $G$ is a minimal $(s, k)$-polar
    obstruction such that $|C| = s$. Thus, since $s+1 \le k$, it is clear that
    $G$ is not an $(s, s+1)$-polar graph. Let $v$ be a vertex of $G$. Clearly,
    if $v \in S$, then $G -v$ is a split graph, hence a $(s, s+1)$-polar graph.
    Additionally, we have from \Cref{lem:Iles-1} that $|I| \le s-1$, so it
    follows from \Cref{theo:DegreeChar(sk)polIPS} that $G-v$ is $(s, s+1)$-polar
    whenever $v \in C$. Notice that, from \Cref{theo:fin(sk)obsminIPSmin3}, we
    have that $|I| < k-1$. Thus, it follows from \Cref{lem:ChracterizationCvC'v}
    and \Cref{theo:DegreeChar(sk)polIPS} that, if $v \in I$, then $G-v$ also is
    an $(s, s+1)$-polar graph. Therefore, $G$ is not an $(s, s+1)$-polar graph
    but any vertex deleted subgraph of $G$ is, so we have that $G$ is a minimal
    $(s, s+1)$-polar obstruction. The assertion follow easily from here

    The last statement is an immediate consequence of items 1 and 2.
\end{proof}

\begin{corollary} \label{cor:psSinftyObsmin}
    Let $s$ be an integer, $s \ge 3$. Any pseudo-split minimal $(s,
    \infty)$-polar obstruction has order at most $2s + 4$, and the bound is
    tight. In consequence, there are finitely many minimal $(s, \infty)$-polar
    obstructions.
\end{corollary}

\begin{proof}
    Let $G$ be a pseudo-split minimal $(s, \infty)$-polar obstruction. We have
    from \Cref{lem:sInftyIFFss+1} that $G$ is a minimal $(s, s+1)$-polar
    obstruction so, by \Cref{theo:fin(sk)obsminIPSmin3}, the order of $G$ is at
    most $2s +4$. The bound is tight because $F_s$ is a pseudo-split minimal
    $(s, \infty)$-polar obstruction with $2s + 4$ vertices.
\end{proof}

% % % % % % % % % % % % % % % % % % % % % % % % % % % % % % % % % % % % % % % %
\subsection{Colorings of pseudo-split graphs} \label{sec:colPS}
% % % % % % % % % % % % % % % % % % % % % % % % % % % % % % % % % % % % % % % %

Brandstädt \cite{brandstadDM186, brandstadDM152, brandstadDAM89} introduced the
$(k, \ell)$-colorings in the 1990s, when he proved that the problem of deciding
whether a graph admits a $(k, \ell)$-coloring is polynomial time solvable if $k,
\ell \le 2$, and NP-complete otherwise.

In this brief section we study some coloring parameters of pseudo-split graphs,
including $(k,\ell)$-colorings, co-chromatic number, and bi-chromatic number.

\begin{lemma} \label{lem:kColPS}
    Let $G$ be a pseudo-split graph, and let $k$ be an integer, $k \ge 2$. Then,
    $G$ is $k$-colorable if and only if $G$ is a $(K_{k+1}, C_5 \oplus
    K_{k-2})$-free graph.
\end{lemma}

\begin{proof}
    It is easy to verify that both, $K_{k+1}$ and $C_5 \oplus K_{k-2}$, are not
    $k$-colorable graphs. Thus, since being $k$-colorable is a hereditary
    property, it follows that any $k$-colorable graph is $(K_{k+1}, C_5 \oplus
    K_{k-2})$-free.

    Conversely, suppose that $G$ is a $(K_{k+1}, C_5 \oplus K_{k-2})$-free
    pseudo-split graph, and let $(C, S, I)$ be a pseudo-split partition of $G$.
    Since split graphs are perfect, if $S = \varnothing$, then $G$ is a
    $K_{k+1}$-free perfect graph, hence a $k$-colorable graph. Otherwise, $G$ is
    an imperfect pseudo-split graph such that $|C| \le k-3$, or $G$ would have
    $C_5 \oplus K_{k-2}$ as an induced subgraph; in this case a proper
    $k$-coloring of $G$ could be obtained assign colors $1, \dots, k-3$ to the
    vertices of $C$, coloring $S$ in a proper way with colors $k-2, k-1$ and
    $k$, and assigning color $k$ to every vertex of $I$.
\end{proof}

\begin{theorem} \label{theo:kEllIPS}
    Let $G$ be an imperfect pseudo-split graph with pseudo-split partition $(C,
    S, I)$, and let $k$ and $\ell$ be nonnegative integers. The following
    statements hold true.
    \begin{enumerate}
        \item $G$ is a $(k,0)$-graph if and only if $|C| \le k-3$;
        
        \item $G$ is a $(0, \ell)$-graph if and only if $|I| \le \ell -3$;
        
        \item $G$ is not a $(1,1)$-graph;

        \item If $k$ and $\ell$ are positive integers, and $k + \ell \ge 3$,
        then $G$ is a $(k, \ell)$-graph.
    \end{enumerate}
    Particularly, $\chi(G) = |C|+3$ and $\theta(G) = |I| + 3$.
\end{theorem}

\begin{proof}
    It follows from \Cref{lem:kColPS} that, for any integer $k \ge 2$, an
    imperfect pseudo-split graph is $k$-colorable if and only if it is a $C_5
    \oplus K_{k-2}$-free graph. Thus, $G$ is a $(k, 0)$-graph if and only if
    $|C| < k-2$.
    
    The second item follows from the first one since a graph $G$ is $(0,
    \ell)$-colorable if and only if $\overline{G}$ is $(\ell, 0)$-colorable, and
    the complement of an imperfect pseudo-split graph is also an imperfect
    pseudo-split graph. Item 3 is due to $G$ has an induced $C_5$, and $C_5$ is
    not a $(1, 1)$-graph.

    Notice that for the last item it is enough to prove that $G$ is a $(1,
    2)$-graph, but this is trivially true since, for any $(1, 2)$-coloring $(A,
    B)$ of $G[S]$, $(A \cup I, B \cup C)$ is a $(1,2)$-coloring of $G$. The last
    statement is a direct consequence of the first two items, although it is
    also a direct consequence of \Cref{theo:recognitionIPS}.
\end{proof}

\begin{corollary} \label{cor:chiCIPS}
    If $G$ is an imperfect pseudo-split graph, then $\chi^c(G) = 3$.
\end{corollary}

\begin{proof}
    As we proved in \Cref{theo:kEllIPS}, $G$ is a $(1, 2)$-colorable graph,
    hence a $3$-cocolorable graph. Thus $\chi^c(G) \ge 3$. In addition, $C_5$ is
    an induced subgraph of $G$, but it is not a $2$-cocolorable graph, so
    $\chi^c(G) > 2$, and the result follows.
\end{proof}

\begin{lemma} \label{lem:subsBicol}
    Let $z$ be a positive integer, and let $\mathcal{F}^b(z)$ be the set of
    minimal $z$-bicolorable obstructions. Then,
    \[ \mathcal{F}^b(z) \subseteq \bigcup_{i=0}^z \mathcal{F}(i, z-i), \]
    where $\mathcal{F}(k, \ell)$ stands for the set of minimal $(k,
    \ell)$-obstructions.
\end{lemma}

\begin{proof}
    Notice that a graph $H$ is a minimal $z$-bicolorable obstruction if and only
    if there exists an integer $i$ with $0 \le i \le z$, such that $H$ contains
    a graph $F_i \in \mathcal{F}(i,z-i)$ as an induced subgraph and, for any
    vertex $v$ of $H$ and every integer $j \in \set{0, \dots, z}$, $H-v$ does not
    contain any graph of $\mathcal{F}(j,z-j)$ as an induced subgraph.
    Particularly, for some integer $i \in \set{0, \dots, z}$, $H$ is not an $(i,
    z-i)$-graph, but every vertex-deleted subgraph of $H$ is, so $H$ is a
    minimal $(i, z-i)$-obstruction.
\end{proof}

\begin{theorem}
    Let $z$ be a positive integer, and let $k$ and $\ell$ be nonnegative
    integers. Let $\mathcal{F}^b_{ps}(z)$ be the set of pseudo-split minimal
    $z$-bicolorable obstructions, and let $\mathcal{F}_{ps}(k, \ell)$ be the set
    of pseudo-split minimal $(k, \ell)$-obstructions. Then,
    \[ \mathcal{F}_{ps}^b(1) = \set{K_2, \overline{K_2}}, \hspace{1em}
    \mathcal{F}_{ps}^b(2) = \set{K_3, C_5, \overline{K_3}}, \]
    and for any integer $z$ with $z \ge 3$,
    \[ \mathcal{F}_{ps}^b(z) = \set{K_{z+1}, C_5 \oplus K_{z-2}, \overline{K_{z+1}},
    C_5 + \overline{K_{z-2}}}. \]
\end{theorem}

\begin{proof}
    It follows from \Cref{theo:kEllIPS} that $\mathcal{F}_{ps}(1,0) = \set{K_2},
    \mathcal{F}_{ps}(1,1) = \set{C_5}$, and $\mathcal{F}_{ps}(k, 0) = \set{K_{k+1},
    C_5 \oplus K_{k-2}}$  for any integer $k \ge 2$. In addition, it also
    follows from \Cref{theo:kEllIPS} that, for any positive integers $k$ and
    $\ell$ with $k + \ell \ge 3$, $\mathcal{F}_{ps}(k, \ell) = \varnothing$.

    Then, since $H \in \mathcal{F}_{ps}(k, \ell)$ if and only if $\overline{H}
    \in \mathcal{F}_{ps}(\ell, k)$, we have from \Cref{lem:subsBicol} and the
    observations in the above paragraph that, for any positive integer $z$,
    \[ \mathcal{F}^b_{ps}(z) = \bigcup_{i=0}^z \mathcal{F}_{ps}(i,z-i), \]
    where the result follows.
\end{proof}

\begin{corollary} \label{cor:chiBIPS}
    Let $G$ be an imperfect pseudo-split graph with pseudo-split partition $(C,
    S, I)$. Then $\chi^b(G) = \max\set{|C|+3, |I|+3} = \max \set{\chi(G),
    \theta(G)}$. Particularly, $G$ is not a $2$-bicolorable graph.
\end{corollary}

\begin{corollary} \label{cor:cochBichPS}%
  Chromatic and bichromatic numbers can be determined on imperfect pseudo-split
  graphs in $O(|V|)$-time  from their degree sequences. Additionally, the
  cochromatic number of imperfect pseudo-split graphs can be determined in
  constant time.
\end{corollary}

\begin{proof}
    From \Cref{theo:recognitionIPS}, we have that $\chi(G)$ and $\theta(G)$ can
    be computed in $O(|V|)$-time from the degree sequence of an imperfect
    pseudo-split graph. Then, we have from \Cref{cor:chiBIPS} that also the
    cochromatic number of imperfect pseudo-split graphs can be computed in
    $O(|V|)$-time. The last part of the statement follows from
    \Cref{cor:chiCIPS}.
\end{proof}

% % % % % % % % % % % % % % % % % % % % % % % % % % % % % % % % % % % % % % % %
\section{$2K_2$- and $C_4$-split graphs} \label{chap:2K2split}
% % % % % % % % % % % % % % % % % % % % % % % % % % % % % % % % % % % % % % % %

\Cref{theo:charHsplit} provide us of a characterization of $2K_2$-split graphs
based on their degree sequences. Next, we characterize $2K_2$-split graphs by
their forbidden induced subgraphs.

\begin{theorem}
    A graph $G$ is a $2K_2$-split graph if and only if $G$ has not induced
    subgraphs isomorphic to the graphs depicted in \Cref{fig:2K2splitObsmin}.
\end{theorem}

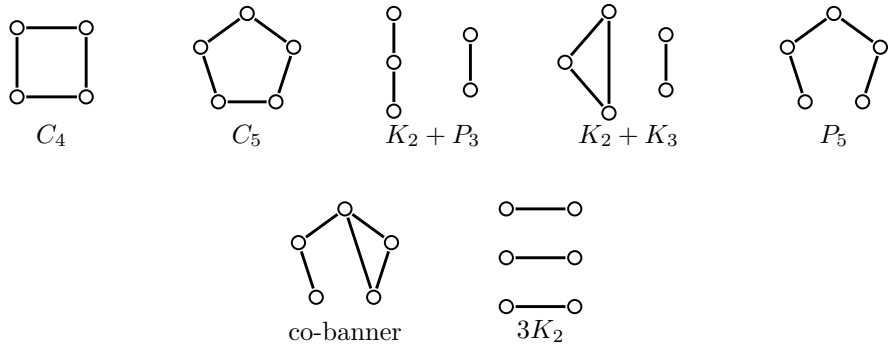
\begin{figure}[!ht]
    \centering
    \begin{tikzpicture}[scale=0.65]

    % % % % % % % % % % % % %	C_4
    \begin{scope}[scale=1, xshift=0cm, yshift=0cm]
    \foreach \i in {0,...,3}
    \node (\i) [vertex] at ({(\i*90)+45}:1){};

    \node (x) [rectangle] at (0,-1.5){$C_4$};
    \end{scope}

    \foreach \i in {0,...,3}
    \draw let \n1={int(mod(\i+1,4))} in [edge] (\i) to node [above] {} (\n1);

    % % % % % % % % % % % % %	C_5
    \begin{scope}[scale=1, xshift=4cm, yshift=0cm]
    \foreach \i in {0,...,4}
    \node (\i) [vertex] at ({(\i*72)+90}:1){};

    \node (x) [rectangle] at (0,-1.5){$C_5$};
    \end{scope}

    \foreach \i in {0,...,4}
    \draw let \n1={int(mod(\i+1,5))} in [edge] (\i) to node [above] {} (\n1);

    % % % % % % % % % % % % %	K_2+P_3
    \begin{scope}[scale=1, xshift=8cm, yshift=0cm]
    \node (0) [vertex] at (-1,0){};
    \node (1) [vertex] at (-1,1){};
    \node (2) [vertex] at (-1,-1){};

    \node (3) [vertex] at (45:0.8){};
    \node (4) [vertex] at (315:0.8){};

    \node (x) [rectangle] at (-0.2,-1.5){$K_2+P_3$};
    \end{scope}

    \draw [edge] (0) to node [above] {} (1);
    \draw [edge] (0) to node [above] {} (2);

    \draw [edge] (3) to node [above] {} (4);

    % % % % % % % % % % % % %	K_2+K_3
    \begin{scope}[scale=1, xshift=12cm, yshift=0cm]
    \node (0) [vertex] at (180:1.5){};
    \node (1) [vertex] at (120:1.2){};
    \node (2) [vertex] at (240:1.2){};

    \node (3) [vertex] at (45:0.8){};
    \node (4) [vertex] at (315:0.8){};

    \node (x) [rectangle] at (-0.2,-1.5){$K_2+K_3$};
    \end{scope}

    \draw [edge] (0) to node [above] {} (1);
    \draw [edge] (0) to node [above] {} (2);
    \draw [edge] (1) to node [above] {} (2);

    \draw [edge] (3) to node [above] {} (4);

    % % % % % % % % % % % % %	P_5
    \begin{scope}[scale=1, xshift=16cm, yshift=0cm]
    \foreach \i in {0,...,4}
    \node (\i) [vertex] at ({(\i*72)+90}:1){};

    \node (x) [rectangle] at (0,-1.5){$P_5$};
    \end{scope}

    \foreach \i in {0,1,3,4}
    \draw let \n1={int(mod(\i+1,5))} in [edge] (\i) to node [above] {} (\n1);

    % % % % % % % % % % % % %	co-banner
    \begin{scope}[scale=1, xshift=6cm, yshift=-4cm]
    \foreach \i in {0,...,4}
    \node (\i) [vertex] at ({(\i*72)+90}:1){};

    \node (x) [rectangle] at (0,-1.5){co-banner};
    \end{scope}

    \foreach \i in {0,1,3,4}
    \draw let \n1={int(mod(\i+1,5))} in [edge] (\i) to node [above] {} (\n1);
    \draw [edge] (3) to node [above] {} (0);

    % % % % % % % % % % % % %	3K_2
    \begin{scope}[scale=1, xshift=10cm, yshift=-4cm]
    \node (0) [vertex] at (-0.7,1){};
    \node (1) [vertex] at (-0.7,0){};
    \node (2) [vertex] at (-0.7,-1){};

    \node (3) [vertex] at (0.7,1){};
    \node (4) [vertex] at (0.7,0){};
    \node (5) [vertex] at (0.7,-1){};

    \node (x) [rectangle] at (0,-1.5){$3K_2$};
    \end{scope}

    \draw [edge] (0) to node [above] {} (3);
    \draw [edge] (1) to node [above] {} (4);
    \draw [edge] (2) to node [above] {} (5);

    \end{tikzpicture}
    \caption{Minimal $2K_2$-split obstructions.}
    \label{fig:2K2splitObsmin}
\end{figure}

\begin{proof}
    It is a routine job to check that any graph in \Cref{fig:2K2splitObsmin} is
    a minimal $2K_2$-split obstruction. To prove the necessary condition, let
    $G$ be a graph without induced subgraphs isomorphic to the graphs in
    \Cref{fig:2K2splitObsmin}. If $G$ is $2K_2$-free, then $G$ is a split graph,
    and hence a $2K_2$-split graph, so let us assume that $G$ has an induced copy
    of $2K_2$ with vertex set $S$. Notice that any vertex in $V_G \setminus S$
    is either completely adjacent to $S$ or completely nonadjacent to it, or $G$
    would have some of the forbidden induced subgraphs.

    Since $G$ is a $3K_2$-free graph, any two vertices $u,v \in V_G \setminus S$
    that are completely nonadjacent to $S$ must be nonadjacent to each other. In
    addition, if there exist two nonadjacent vertices $u,v \in V_G \setminus S$
    that are completely adjacent to $S$, then $G$ has an induced $C_4$, which is
    impossible.

    Hence, if we set $C$ the set of all vertices in $V_G \setminus S$ that are
    completely adjacent to $S$, and $I$ the set of vertices in $V_G \setminus S$
    that are completely nonadjacent to $S$, then $(C, S, I)$ is a $2K_2$-split
    partition of $G$.
\end{proof}

In this section we study polarity on $2K_2$-split graphs. Observe that a graph
is $2K_2$-split if and only if its complement is $C_4$-split. Thus, since the
complement of an $(s, k)$-polar graph is a $(k, s)$-polar graph, we have that by
simple arguments of complements any result about $(s, k)$-polarity on
$2K_2$-split graphs is equivalent to a dual result on $C_4$-split graphs. As the
reader will notice, although some results are similar to those proved in
\Cref{sec:PS} for pseudo-split graphs, there are also remarkable differences.

% % % % % % % % % % % % % % % % % % % % % % % % % % % % % % % % % % % % % % % %
\subsection{Polarity on $2K_2$-split graphs}
% % % % % % % % % % % % % % % % % % % % % % % % % % % % % % % % % % % % % % % %

As in the case of pseudo-split graphs, any $2K_2$-split graph is polar.
Moreover, if $G = (C, S, I)$ is a $2K_2$-split graph, then $(C, S \cup I)$ is a
unipolar partition of $G$, so $G$ is unipolar, and hence polar. Similarly, any
$C_4$-split graph $G = (C, S, I)$ has a unipolar partition, namely $(C \cup
\set{u,v}, I \cup (S \setminus \set{u,v}))$, where $u$ and $v$ are two adjacent
vertices of $S$. Thus, any $C_4$-split graph is unipolar.

With the next simple propositions we completely characterize the $2K_2$- and
$C_4$-split graphs that admit a $(1,k)$-polar partition for any value of $k$.
Observe that, from such results, it follows that $(1,k)$-polarity (including
monopolarity) can be checked in $O(|V|)$-time from the degree sequence of any
$2K_2$- or $C_4$-split graph. We start with an observation that is to
$2K_2$-split graphs as Remark \ref{rem:degrees} is to pseudo-split graphs.

\begin{remark} \label{rem:degrees2K2}%
  Let $G = (C, S, I)$ be a strict $2K_2$-split graph, and let $c=|C|$ and
  $i=|I|$. If $u \in V_G$, then $u \in C$ if and only if $\Deg{u} \ge c+3$, $u
  \in S$ if and only if $\Deg{u} = c+1$, and $u \in I$ if and only if $\Deg{u}
  \le c$. Moreover, $\Deg{u} = c+3$ if and only if $u$ is a vertex in $C$ that
  is completely nonadjacent to $I$, and $\Deg{u} = c$ if and only if $u$ is a
  vertex in $I$ that is completely adjacent to $C$.
\end{remark}

\begin{theorem} \label{theo:2K2(1k)obsmin}%
	Let $k$ be an integer, $k \ge 2$, and let $G = (C, S, I)$ be a $2K_2$-split
	graph. The following sentences are equivalent.
    \begin{enumerate}
        \item $G$ is a $(1,k)$-polar graph.
        \item $G$ is a $\{K_2 \oplus 2K_2, K_1 \oplus (2K_2 +
        \overline{K_{k-1}})\}$-free graph.
        \item If $G$ is a strict $2K_2$-split graph, then $G$ has at most one
        vertex whose degree is at least $|C|+3$ and, if $\Delta_G \ge |C|+3$,
        then $G$ has at most $k-1$ vertices of degree $|C|$.
    \end{enumerate}
	In consequence, the only $2K_2$-split minimal monopolar obstruction is $K_2
	\oplus 2K_2$, and $G$ is monopolar if and only if it has at most one vertex
	whose degree is at least $|C|+3$.
\end{theorem}

\begin{proof}
  Te equivalence between the second and the third item follows easily from
	\Cref{rem:degrees2K2}, so we only prove the equivalence between 1 and 2. It is
	a routine to prove that both, $K_2 \oplus 2K_2$ and $K_1 \oplus(2K_2 +
	(k-1)K_1)$ are $2K_2$-split minimal $(1, k)$-polar obstructions. Hence, every
	$(1,k)$-polar graph is a $\{K_2 \oplus 2K_2, K_1 \oplus (2K_2 +
	\overline{K_{k-1}})\}$-free graph.
    
  For the converse, let us assume that $G = (C, S, I)$ is a $2K_2$-split minimal
	$(1, k)$-polar obstruction; notice that $S \ne \varnothing$, or $G$ would be a
	$(1,1)$-polar graph, and hence a $(1, k)$-polar graph. Also, notice that $G$
	does not have isolated vertices, because if $v$ was a vertex of degree zero in
	$G$, and $(A, B)$ is a $(1, k)$-polar partition of $G-v$, then $(A \cup
	\set{v}, B)$ would be a ($1, k)$-polar partition of $G$, contradicting the
	election of $G$. Particularly, each vertex in $I$ has a neighbor in $C$.

	Now, if $|C| \ge 2$, then $K_2 \oplus 2K_2$ is an induced subgraph of $G$, so
	$G\cong K_2\oplus 2K_2$. In addition, if $C = \varnothing$, then $G$ clearly
	is a $(1, 2)$-polar graph, and hence a $(1,k)$-polar graph, which is
	impossible. Thus, if $G \not\cong K_2\oplus 2K_2$, $|C|=1$. Additionally, if
	$|I| < k-1$, then $(C, S \cup I)$ would be a $(1, k)$-polar partition of $G$,
	but that is absurd, so it must be the case that $|I| \ge k-1$, and it follows
	from the previous observations that $G$ has $K_1 \oplus (2K_2 + (k-1)K_1)$ as
	an induced subgraph, so $G \cong K_1 \oplus (2K_2 + (k-1)K_1)$ by the
	minimality of $G$. Thus, the only minimal $(1,k)$-polar obstructions that are
	$2K_2$-split are precisely $K_2 \oplus 2K_2$ and $K_1 \oplus (2K_2 +
	\overline{K_{k-1}})$, implying that every $2K_2$-split graph that does not
	contain an induced copy of such graphs is $(1,k)$-polar.
\end{proof}

\begin{theorem} \label{theo:C4splitMonop}%
  Let $k$ be an integer, $k \ge 2$, and let $G = (C, S, I)$ be a $C_4$-split
graph. The following sentences are equivalent.
  \begin{enumerate}
      \item $G$ is a $(1,k)$-polar graph.
      \item $G$ is a $K_1 \oplus C_4$-free
      graph.
      \item If $G$ is a strict $C_4$-split graph, then $\Delta_G \le 2$.
  \end{enumerate}
In consequence, the only $C_4$-split minimal monopolar obstruction is $K_1
\oplus C_4$, and $G$ is monopolar if and only if $\Delta_G \le 2$.
\end{theorem}

\begin{proof}
	We only prove the equivalence between 1 and 2. It is a routine to show that
	$C_4 \oplus K_1$ is a $(1, \infty)$-polar obstruction such that any
	vertex-deleted subgraph is $(1, 2)$-polar, so we have that $C_4 \oplus K_1$ is
	a minimal $(1, k)$-polar obstruction for any integer $k \ge 2$.
	
	Now, let $G = (C, S, I)$ be a $C_4$-split graph. If $C = \varnothing$, for any
	two nonadjacent vertices $u,v \in S$, $(I \cup (S \setminus \set{u,v}),
	\set{u,v})$ is a $(1,2)$-polar partition of $G$, so $G$ is $(1, k)$-polar.
	Otherwise, we have that $|C| \ge 1$, so $C_4 \oplus K_1$ is an induced
	subgraph of $G$ and $G$ is not a $(1, k)$-polar graph.
\end{proof}

% % % % % % % % % % % % % % % % % % % % % % % % % % % % % % % % % % % % % % % %
\subsubsection{$2K_2$-split minimal $(s, k)$-polar obstructions}
% % % % % % % % % % % % % % % % % % % % % % % % % % % % % % % % % % % % % % % %

In this section, we provide upper bounds for the order of minimal $2K_2$-split
minimal $(s,k)$-polar obstructions. We start with a complete characterization of
$2K_2$-split graphs that admit an $(s, k)$-polar partition.

\begin{theorem} \label{thm:2K2Char(sk)pol} %
  Let $s$ and $k$ be integers, $s, k\ge 2$, and let $G = (C, S, I)$ be a strict
	$2K_2$-split graph. Let $c$ and $i$ be the cardinalities of $C$ and $I$,
	respectively. The following statements hold true.
	\begin{enumerate}
		\item If $s \ge c$ and $k \ge i+2$, then $G$ is an $(s, k)$-polar graph.

		\item If $s \ge c+2$ and $k \ge i+1$, then $G$ is an $(s, k)$-polar graph.

		\item If $s \le c-1$ and $k \le i$, then $G$ is not an $(s, k)$-polar
		graph.

		\item If $s \le c-1$ and $k \ge i+1$, then $G$ is an $(s, k)$-polar graph
		if and only if there is a subset $C'$ of $C$ with at least $c-s+2$
		vertices that is completely nonadjacent to $I$.

		\item If $s \ge c+1$ and $k \le i$, then $G$ is an $(s, k)$-polar graph
		if and only if there exists a subset $I'$ of $I$ with at least $i-k+2$
		vertices that satisfies some of the following conditions:
		\begin{enumerate}
			\item $I'$ is completely adjacent to $C$.

			\item There exists a vertex $v \in C$ such that $I'$ is completely
			adjacent to $C \setminus \set{v}$ and $v$ is completely
			nonadjacent to $I'$.
		\end{enumerate}

		\item If $s = c$ and $k \le i$, then $G$ is an $(s, k)$-polar graph if
		and only if there exists a subset $I'$ of $I$ with at least $i-k+2$
		vertices and a vertex $v \in C$ such that $I'$ is completely adjacent to
		$C \setminus \set{v}$ and $v$ is completely nonadjacent to $I'$.

		\item If $s = c$ and $k = i+1$, then $G$ is an $(s, k)$-polar graph if
		and only if some of the following sentences is satisfied:
		\begin{enumerate}
			\item there exists a subset $C'$ of $C$ with at least $c-s+2$
			vertices that is completely nonadjacent to $I$.

			\item there exists a nonempty subset $I'$ of $I$ and a vertex $v \in
			C$ such that $I'$ is completely adjacent to $C \setminus \set{v}$ and
			$v$ is completely nonadjacent to $I'$.
		\end{enumerate}

		\item If $s = c+1$ and $k = i+1$, then $G$ is an $(s, k)$-polar graph if
		and only if some of the following sentences is satisfied:
		\begin{enumerate}
			\item there exists a nonempty subset $C'$ of $C$ that is completely
			nonadjacent to $I$.

			\item there exists a nonempty subset $I'$ of $I$ such that satisfies
			some of the following conditions:
			\begin{enumerate}
				\item $I'$ is completely adjacent to $C$.

				\item There is a vertex $v \in C$ that is completely
				nonadjacent to $I'$ and such that $I'$ is completely adjacent to
				$C \setminus \set{v}$.
			\end{enumerate}
		\end{enumerate}
		
	\end{enumerate}
\end{theorem}

\begin{proof}
	Let $S=\set{u,v,x,y}$ and assume that $uv,xy\in E_G$.
	\begin{enumerate}
		\item It is enough to notice that $(C, S\cup I)$ is a $(c,i+2)$-polar
		partition of $G$.

		\item It is enough to notice that $(C\cup\set{u,v}, I\cup\set{x,y})$ is a
		$(c+2,i+1)$-polar partition of $G$.

		\item Assume for a contradiction that $G$ admits an $(s, k)$-polar partition
		$(A,B)$. Notice that $C\not\subseteq A$, so $C\cap B\ne \varnothing$, in
		which case the vertices of one component of $G[S]$ are in $A$. Suppose
		without loss of generality that $u,v\in A$. Then $I\cup\set{x,y}\subseteq
		B$, but in such a case $G[B]$ has at least $i+1$ components, a
		contradiction.

		\item For the necessary condition it is enough to notice that
		$(\set{u,v}\cup C\setminus C', \set{x,y}\cup C'\cup I)$ is an
		$(s,i+1)$-polar partition of $G$. For the sufficient condition, let us
		assume that $(A, B)$ is a polar partition of $G$. Since $s \le c-1$, $C \cap
		B \ne \varnothing$. Thus, we can assume without loss of generality that
		$\set{u,v} \in A$ and $\set{x,y} \in B$, and we have that $I \subseteq B$.
		Let $C' = C \cap B$. Hence, $C'$ is completely nonadjacent to $I$ and $|C'|
		\ge c-s+2$, so the result follows.

		\item For the necessary condition, notice that $(C \cup I', S \cup I
		\setminus I')$ is an $(c+1, k)$-polar partition of $G$. For the sufficient
		condition let $(A, B)$ be an $(s, k)$-polar partition of $G$. Since $2K_2$
		is not a complete multipartite graph, we have that $B \cap S \ne
		\varnothing$, which implies that $I \cap A \ne \varnothing$, and then both,
		$B \cap \set{u,v} \ne \varnothing$ and $B \cap \set{x,y} \ne \varnothing$.
		In consequence $C \subseteq A$. Set $I' = I \cap A$, and notice that, since
		$k \le i$, then $|I'| \ge i-k+2$. Moreover, since $G[A]$ is a complete
		$s$-partite graph and $I' \cup C \subseteq A$ we have that every vertex of
		$I'$ is either completely adjacent to $C$ or it is adjacent to any vertex of
		$C$ except for a vertex $v$. In addition, if a vertex $w \in C$ is adjacent
		to a vertex $z \in I'$, then $w$ is completely adjacent to $I'$. The
		statement easily follows form the above observations.

		\item For the necessary condition, notice that $(C\cup I', S\cup I \setminus
		I')$ is a $(c,k)$-polar partition of $G$. For the sufficient condition let
		$(A,B)$ be an $(s, k)$-polar partition of $G$. Since $2K_2$ is not a
		complete multipartite graph, $B \cap S \ne \varnothing$, which implies that
		$I \cap A \ne \varnothing$. But then, $B \cap \set{u,v} \ne \varnothing$ and
		$B \cap \set{x,y} \ne \varnothing$, implying that $C \subseteq A$. Set
		$I'=I\cap A$. Since $k\le i$, $|I'|\ge i-k+2$. Moreover, since $G[A]$ is a
		complete $c$-partite graph and $I'\cup C\subseteq A$ we have that every
		vertex of $I'$ is adjacent to any vertex of $C$ except for a vertex $v$. In
		addition, if a vertex $w\in C$ is adjacent to a vertex $z\in I'$, then $w$
		is completely adjacent to $I'$. The conclusion follows easily from here.

		\item For the necessary condition we only have to notice that, in case of
		$(a)$ occurs, $(\set{u,v} \cup C \setminus C', \set{x,y} \cup C' \cup I)$ is
		an $(s, k)$-polar partition of $G$ while, if $(b)$ occurs, then $(C \cup I',
		S \cup I \setminus I')$ is an $(s, k)$-polar partition of $G$. For the
		sufficient condition let $(A, B)$ be an $(s, k)$-polar partition of $G$. If
		$A \cap S \ne \varnothing$, then $B \cap C \ne \varnothing$, which implies,
		without loss of generality, that $\set{u,v} \subseteq A$ and $\set{x,y}
		\subseteq B$, and therefore $I \subseteq B$. Then, if $C' = C \cap B$, we
		have that $|C'| \ge c-s+2$ and $C'$ is completely nonadjacent to $I$.
		Otherwise, if $A \cap S = \varnothing$, $S \subseteq B$, and therefore $C
		\subseteq A$. Notice that, since $k = i+1$, $A \cap I \ne \varnothing$. Let
		$I' = A \cap I$. Since $C \cup I'\subseteq A$ and $A$ induces a complete
		$c$-partite graph, there exist a vertex $v \in C$ such that $I'$ is
		completely adjacent to $C\setminus\set{v}$ but $v$ is completely nonadjacent
		to $I'$.

		\item For the necessary condition, notice that in case that $(a)$ occurs,
		$(\set{u, v} \cup C \setminus C', \set{x, y} \cup C' \cup I)$ is an $(s,
		k)$-polar partition of $G$ while, if $(b)$ occurs, then $(C \cup I', S \cup
		I \setminus I')$ is an $(s, k)$-polar partition of $G$. For the sufficient
		condition, let us consider an $(s, k)$-polar partition $(A, B)$ of $G$. If
		$G[B \cap S]$ is connected, then the vertex set of one of the connected
		components of $G[S]$ is completely contained in $A$, let us say, without
		loss of generality, that $\set{u ,v} \subseteq A$. Observe that, in this
		case, $C \not \subseteq A$, so $C ' = C \cap B$ is not empty. In addition,
		$I \cap A = \varnothing$, so $I \subseteq B$. Thus, since $G[B]$ is a
		$P_3$-free graph, $C'$ is completely nonadjacent to $I$, and we have the
		case $(a)$ of the statement. Otherwise, $G[B]$ is disconnected, which
		implies that $C \cap B = \varnothing$ and $I \not \subseteq B$. Then, we
		have that $C \subseteq A$ and the set $I' = I \cap A$ is not empty. Hence,
		$C \cup I' \subseteq A$ and, since $G[A]$ is $\overline{P_3}$-free, we have
		that either $I'$ is completely adjacent to $C$, or there is a vertex $v \in
		C$ such that $I'$ is completely adjacent to $C \setminus \set{v}$ and $v$ is
		completely nonadjacent to $I'$, so item $(b)$ of the statement follows.
	\end{enumerate}
\end{proof}

The following propositions are consequences of \Cref{thm:2K2Char(sk)pol}. They
are intended to prove an upper bound for the order of $2K_2$-split minimal $(s,
k)$-polar obstructions for arbitrary integers $s$ and $k$.

Observe that, for a $2K_2$-split graph $G = (C, S, I)$, if some of $C, S$ or $I$
is an empty set, then $G$ is a $2$-polar graph. Hence, for any integers $s$ and
$k$ with $s, k \ge 2$, any $2K_2$-split minimal $(s, k)$-polar obstruction $G =
(C, S, I)$ is such that $C, S$ and $I$ are all of them nonempty sets. We will
use this observation in the following proofs without any explicit mention.

The following proposition is a direct consequence of item 3 of
\Cref{thm:2K2Char(sk)pol}.

\begin{lemma} \label{lem:2K2obsminBound3}%
  Let $s$ and $k$ be integers, $s, k \ge 2$, and let $G = (C, S, I)$ be a strict
	$2K_2$-split graph. The following assertions hold for $c = |C|$ and $i = |I|$.
	\begin{enumerate}
		\item If $c \ge s+2$ and $i \ge k$, for each vertex $v \in C$, $G-v$ is
		not an $(s, k)$-polar graph, so $G$ is not a minimal $(s, k)$-polar
		obstruction.

		\item If $c \ge s+1$ and $i \ge k+1$, for each vertex $v \in I$, $G-v$
		is not an $(s, k)$-polar graph, so $G$ is not a minimal $(s, k)$-polar
		obstruction.
	\end{enumerate}
\end{lemma}

\begin{lemma} \label{lem:2K2obsminBound4}%
	Let $s$ and $k$ be integers, $s, k \ge 2$, and let $G = (C, S, I)$ be a
	$2K_2$-split graph. If $|C| \ge s+2$ and $|I| \le k-1$, then $G$ is not a
	minimal $(s, k)$-polar obstruction.
\end{lemma}

\begin{proof}
  Let $c = |C|$ and $i = |I|$. Suppose for a contradiction that $G$ is a minimal
	$(s, k)$-polar obstruction. Notice that, for each vertex $v \in C$, $(C
	\setminus \set{v}, S, I)$ is the $2K_2$-split partition of $G-v$ and $|C
	\setminus \set{v}| = c-1 \ge s+1$.

	Since $G$ is a minimal $(s, k)$-polar obstruction, for each $v \in C$, the
	graph $G-v$ is an $(s, k)$-polar graph, which implies, by item 4 of
	\Cref{thm:2K2Char(sk)pol}, that there is a subset $C'_v$ of $C \setminus
	\set{v}$ with at least $c-s+1$ vertices that is completely nonadjacent to
	$I$. In addition, also by item 4 of \Cref{thm:2K2Char(sk)pol}, since $G$ is
	not an $(s, k)$-polar graph, each vertex $v \in C$ is adjacent to at least
	one vertex in $I$.

	Let $H$ be a graph obtained from $G$ by deleting $c-s-1$ vertices of $C$.
	Then $H$ has a $2K_2$-split partition $(C^\ast, S, I)$, with $|C^\ast| =
	s+1$. Notice that each $v \in C^\ast \subseteq C$ has at least one neighbor
	in $I$, which implies that the only subset $C'$ of $C^\ast$ that is
	completely nonadjacent to $I$ is the empty set. Thus, we have from item 4 of
	\Cref{thm:2K2Char(sk)pol} that $H$ is nos an $(s, k)$-polar graph, but that
	is impossible since $H$ is a proper induced subgraph of $G$, which is by
	assumption a minimal $(s, k)$-polar obstruction.
\end{proof}

Next, we identify some particular minimal $(s, k)$-polar obstructions.

\begin{remark} \label{rem:2K2some(sk)obsmin}
	Let $s$ and $k$ be integers, $s \ge 2$.
	\begin{enumerate}
		\item The strict $2K_2$-split graph $G = (C, S, I)$ such that $|C| = s$,
		$|I| = 1$, and $C$ is completely adjacent to $I$, is a minimal $(s,
		2)$-polar obstruction.

		\item Let $k \ge 2$, and let $G = (C, S ,I)$ be the strict $2K_2$-split
		graph such that $|C| = s+1$, $|I| = 1$, and for two vertices $u$ and $v$
		in $C$, $C' = C \setminus \{u,v\}$ is completely adjacent to $I$, and
		$\set{u,v}$ is completely nonadjacent to $I$. Then, $G$ is a minimal
		$(s, k)$-polar obstruction.

		\item Let $k \ge 3$, and let $G = (C, S, I)$ be the strict $2K_2$-split
		graph such that $|C| = s+1$, $|I| = 1$, and for a vertex $u \in C$, $C'
		= C \setminus \{u\}$ is completely adjacent to $I$, and $u$ is
		completely nonadjacent to $I$. Then, $G$ is a minimal $(s, k)$-polar
		obstruction.
  \end{enumerate}
\end{remark}

\begin{lemma} \label{lem:Improve2K2obsminBound4} %
  Let $s$ and $k$ be integers, $s, k \ge 2$, and let $G = (C, S, I)$ be a
	$2K_2$-split graph. If $|C| = s+1$ and $|I| = k$, then $G$ is not a minimal
	$(s, k)$-polar obstruction.
\end{lemma}

\begin{proof}
	Let $c = |C|$ and $i = |I|$. Assume for a contradiction that $G$ is a minimal
	$(s, k)$-polar obstruction, and for a vertex $v \in C$, let $C' = C \setminus
	\set{v}$. Let $(A, B)$ be an $(s, k)$-polar partition of $G-v$. Since $i = k$
	and $2K_2$ is not a complete multipartite graph, we have that $A \cap I \ne
	\varnothing$, which implies that $B$ is present in both components of $2K_2$,
	and therefore $C' \subseteq A$.

	Let $u \in A \cap I$. Since $C' \subseteq A$ and $|C'| = s$, we have that there
	is a vertex $w \in C'$ such that $C'\setminus\set{w}$ is completely adjacent
	to $u$ and $wu \notin E$. Now, since $i = k \ge 2$, we have that $G[C \cup S
	\cup \set{v}]$ is a proper induced subgraph of $G$ that contains one of the
	minimal $(s, k)$-polar obstruction mentioned in
	\Cref{rem:2K2some(sk)obsmin}, a contradiction.
\end{proof}

\begin{lemma} \label{lem:2K2obsminBound6}%
	Let $s$ and $k$ be integers, $s, k \ge 2$, and let $G = (C, S, I)$ be a
	$2K_2$-split graph. If $|C| = s$ and $|I| \ge 2k-1$, then, $G$ is not a
	minimal $(s, k)$-polar obstruction.
\end{lemma}

\begin{proof}
	Let $c = |C|$ and $i = |I|$. Assume for a contradiction that $G$ is a minimal
	$(s, k)$-polar obstruction, and let $u \in I$. Then, $G-u$ is an $(s,
	k)$-polar graph and, by item 6 of \Cref{thm:2K2Char(sk)pol}, there is a subset
	$I'_u$ of $I \setminus \set{u}$ with at least $i-k+1$ vertices, and a vertex
	$v_u \in C$ such that, $I'_u$ is completely adjacent to $C \setminus
	\set{v_u}$ and $v_u$ is completely nonadjacent to $I'_u$. Now, let $x \in
	I'_u$. By the same argument of the paragraph above, there is a subset $I'_x$
	of $I \setminus \set{x}$ with at least $i-k+1$ vertices, and a vertex $v_x \in
	C$ such that, $I'_x$ is completely adjacent to $C \setminus \set{v_x}$ and
	$v_x$ is completely nonadjacent to $I'_x$.
	% Notice that $|I'_u| \ge i-k+1 \ge k \ge 2$.

	Observe that $2i-2k+2 \ge i+1$, because $i \ge 2k-1$. Thus, we have that
	$I'_x \cap I'_u \ne \varnothing$, otherwise
	\[ i = |I| \ge |I'_x \cup I'_u| \ge 2(i-k+1) \ge i+1, \]
	which is absurd. Since $x \in I'_u \setminus I'_x$ and $I'_x \cap I'_u \ne
	\varnothing$, we have that $|I'_x \cup I'_u| \ge |I'_x| + 1 \ge i-k+2$, $v_u
	= v_x$, $I'_x\cup I'_u$ is completely adjacent to $C \setminus \set{v_u}$, and
	$v_u$ is completely nonadjacent to $I'_x\cup I'_u$. This is impossible,
	since item 6 of \Cref{thm:2K2Char(sk)pol} implies that in such a case $G$ is
	an $(s, k)$-polar graph, contradicting our initial assumption.
\end{proof}

\begin{lemma} \label{lem:2K2obsminBound5}%
	Let $s$ and $k$ be integers, $s, k \ge 2$, and let $G = (C, S, I)$ be a
	$2K_2$-split graph. If $|C| \le s-1$ and $|I| \ge 2k-1$, then $G$ is not a
	minimal $(s, k)$-polar obstruction.
\end{lemma}

\begin{proof}
  Let $c = |C|$ and $i = |I|$. Assume for a contradiction that $G$ is a minimal
	$(s, k)$-polar obstruction, and let $u \in I$. Then $G-u$ is an $(s, k)$-polar
	graph and, by item 5 of \Cref{thm:2K2Char(sk)pol}, there exists a subset
	$I'_u$ of $I \setminus \set{u}$ with at least $i-k+1$ vertices and a vertex
	$v_u \in C$ such that $I'_u$ is completely adjacent to $C \setminus \set{v_u}$
	and, $v_u$ is either completely adjacent or completely nonadjacent to $I'_u$.
	% $|I'_u|= i-k+1\ge k\ge 2$,

	We claim that $v_u$ is completely adjacent to $I'_u$, and we prove it by means
	of contradiction. Let $x \in I'_u$. Then $G-x$ is an $(s, k)$-polar graph and
	again, we have from item 5 of \Cref{thm:2K2Char(sk)pol} that there exists a
	subset $I'_x$ of $I \setminus \set{x}$ with at least $i-k+1$ vertices and a
	vertex $v_x \in C$ such that $I'_x$ is completely adjacent to $C \setminus
	\set{v_x}$ and, $v_x$ is either completely adjacent or completely nonadjacent
	to $I'_x$.
	
	Observe that, as it occurred in \Cref{lem:2K2obsminBound6}, since $i \ge
	2k-1$, there is a vertex $w \in I'_x \cap I'_u$. Since we are assuming $v_u$
	is completely nonadjacent to $I'_u$, we have that $v_u$ is not adjacent to
	$w$, and due to $w \in I'_x$, we have that $v_u = v_x$. But then, $I'_u \cup
	I'_x$ is completely adjacent to $C \setminus v_u$ and $v_u$ is completely
	nonadjacent to $I'_u \cup I'_x$. Moreover, since $x \notin I'_x$, the set
	$I'_u \cup I'_x$ has at least $i-k+2$ vertices. But then, item 5 of
	\Cref{thm:2K2Char(sk)pol} implies that $G$ is an $(s, k)$-polar graph, a
	contradiction. The contradiction arose from assuming that $v_u$ is completely
	nonadjacent to $I'_u$, so it must be the case that, for every vertex $u \in
	I$, there exists a subset $I'_u$ of $I$ with at least $i-k+1$ vertices such
	that $I'_u$ is completely adjacent to $C$.

	But then, for any $x \in I'_u$ and any subset $I'_x$ of $I \setminus \set{x}$
	with $i-k+1$ vertices such that $I'_x$ is completely adjacent to $C$, we
	have that $I'_x \cup I'_u$ is a subset of $I$ with at least $i-k+2$ vertices
	that is completely adjacent to $C$, which implies by
	\Cref{thm:2K2Char(sk)pol} that $G$ is an $(s, k)$-polar graph, contradicting
	our initial assumption.
\end{proof}

Now, we are ready to give an upper bound for the order of the $2K_2$-split
minimal $(s, k)$-polar obstructions.

\begin{theorem} \label{thm:2K2boundSk}%
	Let $s$ and $k$ be integers, $s,k \ge 2$. Every $2K_2$-split minimal $(s,
	k)$-polar obstruction has order at most $s+ 2k + 2$.
\end{theorem}

\begin{proof}
	Let $G = (C, S, I)$ be a $2K_2$-split minimal $(s, k)$-polar obstruction. It
	follows from
	\Cref{lem:2K2obsminBound3,lem:2K2obsminBound4,lem:Improve2K2obsminBound4}
	that, if $|C| \ge s+1$, then $|V_G| \le s + k + 4$. Additionally, we conclude
	from \Cref{lem:2K2obsminBound6,lem:2K2obsminBound5} that $|V_G| \le s + 2k +
	2$ whenever $|C| \le s$. Hence, we have that $|V_G| \le \max \set{s + k + 4, s
	+ 2k + 2}$. However, since $k \ge 2$, we have that $s + 2k + 2 \ge s + k + 4$,
	so the result follows.
\end{proof}

We continue with a characterization for $2K_2$-split $(s, \infty)$-polar graphs,
and then with an upper bound for the order of $2K_2$-split minimal $(s,
\infty)$-minimal obstructions.

\begin{lemma} \label{lem:2K2char(sInfty)polar}%
	Let $s$ be an integer, $s \ge 2$, and let $G = (C, S, I)$ be an strict
	$2K_2$-split graph. Then, $G$ is an $(s, \infty)$-polar graph if and only if
	either $s \ge |C|$ or there is a subset $C'$ of $C$ with at least $|C|-s+2$
	vertices that is completely nonadjacent to $I$.
\end{lemma}

\begin{proof}
	Let $c = |C|$ and $i = |I|$. Suppose that $G$ has an $(s, \infty)$-polar
	partition $(A, B)$. If $c > s$, since $G[A]$ is $K_{s+1}$-free, then $C \not
	\subseteq A$, so $C' = C \cap B \ne \varnothing$. Moreover, $G[B]$ is
	$P_3$-free and $A$ induces a $\overline{P_3}$-free graph, which implies,
	without loss of generality, that $\set{u,v} \in A$ and $\set{x,y} \in B$.
	Thus, $I \cap A = \varnothing$ because $G[A]$ is $\overline{P_3}$-free, so $I
	\subseteq B$, and $|C'| \ge c-s+2$ because $A$ induces a $K_{s+1}$-free graph.
	Additionally, since $C' \cup I \subseteq B$, we have that $C'$ is completely
	nonadjacent to $I$, and we are done. 

	For the converse implication, if $s \le c$, then $(C, S \cup I)$ is an $(s,
	i+2)$-polar partition of $G$. Otherwise, there is a set $C'$ of $C$ with at
	least $c-s+2$ vertices that is completely nonadjacent to $I$. In this case,
	$(\set{u,v} \cup C \setminus C', \set{x,y} \cup C' \cup I)$ is an $(s,
	i+1)$-polar partition of $G$, and the result follows.
\end{proof}

For each integer $s \ge 2$, let $H_s = (C, S, I)$ be the strict $2K_2$-split
graph such that $|C| = s+1$, $|I| = s-1$, and for an injection $f \colon I \to
C$, a vertex $v \in I$ is adjacent to a vertex $u \in C$ if and only if $f(v) =
u$. Notice that, by \Cref{lem:2K2char(sInfty)polar}, $H_s$ is a minimal $(s,
\infty)$-polar obstruction.

\begin{theorem} \label{theo:uBound2K2sSinfty}%
  Let $s$ be an integer, $s \ge 2$. Every $2K_2$-split minimal $(s,
	\infty)$-polar obstruction has order at most $2s + 4$, and the bound is tight.
	
	Consequently, there is only a finite number of $2K_2$-split minimal $(s,
	\infty)$-polar obstructions.
\end{theorem}

\begin{proof}
	Let $G = (C, S, I)$ be a $2K_2$-split minimal $(s, \infty)$-polar obstruction,
	and let $c$ and $i$ be the number of vertices in $C$ and $I$, respectively.
	From \Cref{lem:2K2char(sInfty)polar}, we have that $c > s$. In addition, since
	$G$ is a minimal $(s, k)$-polar obstruction for some positive integer $k$, we
	have from \Cref{lem:2K2obsminBound3,lem:2K2obsminBound4}, that $c \le s+1$, so
	we conclude that $c = s+1$.

	By the minimality of $G$, we have from \Cref{lem:2K2char(sInfty)polar} that,
	for each $u \in I$, there is a subset $C'_u$ of $C$, with at least three
	vertices, that is completely nonadjacent to $I \setminus \set{u}$.
	Additionally, since $G$ does not admit an $(s, k)$-polar partition,
	\Cref{lem:2K2char(sInfty)polar} implies that at most two vertices of $C$ are
	completely nonadjacent to $I$, so each vertex $u \in I$ is adjacent to at
	least one vertex of $C'_u$. Moreover, it follows from the previous
	observations that, for each $u \in I$, there is at least one vertex in
	$C'_u$ that is not in $C'_v$ for any $v \in I \setminus \set{u}$. Therefore,
	$|\bigcup_{u \in I} C'_u| \ge i+2$, so $c \ge i +2$, and it follows that
	$|V_G| = |C| + |S| + |I| \le 2s + 4$.

	The bound is tight since $H_s$ is a $K_2$-split minimal $(s, \infty)$-polar
	obstruction of order $2s+4$.
\end{proof}

Unlike pseudo-split graphs, $2K_2$-split graphs does not constitute a
self-complementary class of graphs, so we cannot use simple arguments of
complements to conclude results for $(\infty, k)$-polarity from those of $(s,
\infty)$-polarity on this class. Next, we provide an upper bound for the order
of $2K_2$-split minimal $(\infty, k)$-minimal obstructions by proving similar
results to \Cref{lem:2K2char(sInfty)polar,theo:uBound2K2sSinfty} for $(\infty,
k)$-polarity on $2K_2$-split graphs.

\begin{lemma} \label{lem:2K2char(inftyK)polar}%
	Let $k$ be an integer, $k \ge 2$, and let $G = (C, S, I)$ be an strict
	$2K_2$-split graph. Then, $G$ is an $(\infty, k)$-polar graph if and only if
	either $|I| \le k-1$ or there exists a subset $I'$ of $I$ with at least
	$|I|-k+2$ vertices such that $G[C \cup I']$ is a complete multipartite graph.
\end{lemma}

\begin{proof}
	Let $c = |C|$ and $i = |I|$. Suppose that $G$ has an $(\infty, k)$-polar
	partition $(A, B)$. Since $G[S ]$ is not a complete multipartite graph, we
	have that $S \not\subseteq A$, so $S \cap B \ne \varnothing$. From here, if $i
	\ge k$, then $I \not\subseteq A$ because $G[B]$ is $(k+1)K_1$-free, so $I' = I
	\cap A \ne \varnothing$. Hence, since $G[A]$ is a $\overline{P_3}$-free graph,
	we have that $A \cap S$ is an independent set, so $B$ intersects the vertex
	sets of both of the connected components of $G[S]$.  But then, $C \cap B =
	\varnothing$, because $B$ induces a $P_3$-free graph. Therefore $C \cup I'
	\subseteq A$, and $C \cup I'$ induces a complete multipartite graph. Notice
	that, due to $G[B]$ is $(k+1)K_1$-free and $B$ intersects the vertex sets of
	both components of $G[S]$, $|I'| \ge i-k+2$.

	For the converse implication, let $S'$ be a maximum clique of $G[S]$. If $i
	\le k-1$, then $(C \cup S', I \cup S \setminus S')$ is an $(\infty,
	k)$-polar partition of $G$. Otherwise, there is a subset $I'$ of $I$ with at
	least $i-k+2$ vertices such that $G[C \cup I']$ is a complete multipartite
	graph, so in this case $(C \cup I', S \cup I \setminus I')$ is an $(\infty,
	k)$-polar partition of $G$.
\end{proof}

\begin{theorem}  \label{theo:uBound2K2sInftyK}
	Let $k$ be an integer, $k \ge 2$. Any $2K_2$-split minimal $(\infty,
	k)$-polar obstruction has order at most $2 + 2k + 2^{2k-1}$. In consequence,
	there is only a finite number of $2K_2$-split minimal $(\infty, k)$-polar
	obstructions.
\end{theorem}

\begin{proof}
	Let $G = (C, S, I)$ be a $2K_2$-split minimal $(\infty, k)$-polar obstruction,
	and let $c$ and $i$ be the number of vertices in $C$ and $I$, respectively.
	From \Cref{lem:2K2char(inftyK)polar} we have that $i \ge k$. Moreover, since
	$G$ is a minimal $(s, k)$-polar obstruction for some positive integer $s$, we
	have from \Cref{lem:2K2obsminBound3,lem:2K2obsminBound6,lem:2K2obsminBound5},
	that $i \le 2k-2$.

	By the minimality of $G$, for each vertex $x \in C$, $G-x = (C \setminus
	\set{x}, S, I)$ is an $(\infty,k)$-polar graph with at least $k$ vertices in
	its stable part, so it follows from \Cref{lem:2K2char(inftyK)polar} that
	there is a subset $I'_x$ of $I$ with at least $i-k+2$ vertices such that
	$G[I'_x \cup C \setminus \set{x}]$ is a complete multipartite graph.

	We claim that, for any two different vertices $u,v \in C$, if $I'_v$ is a
	subset of $I'_u$, then the neighborhood of each vertex in $I'_v$ is
	precisely $C \setminus \set{u,v}$. Notice that this would imply that there are
	not three vertices $u,v,w \in I$ such that $I'_u = I'_v = I'_w$.

	To prove our claim, suppose that $u$ and $v$ are vertices in $C$ such that
	$I'_v \subseteq I'_u$. Since $G[I'_v \cup C \setminus \set{v}]$ is a complete
	multipartite graph we have two possibilities, either $I'_v$ is completely
	adjacent to $C \setminus \set{v}$ or there is a vertex $w \in C \setminus
	\set{v}$ such that $I'_v$ is completely adjacent to $C \setminus \set{v,w}$ and
	$w$ is completely nonadjacent to $I'_v$. Notice that, regardless of the
	case, since $v \in C \setminus \set{u}$ and $G[I'_u \cup C \setminus \set{u}]$
	is $\overline{P_3}$, $v$ is either completely adjacent or completely
	nonadjacent to $I'_u$, and therefore, $v$ is either completely adjacent or
	completely nonadjacent to $I'_v$. But we have from the previous observation
	that, if $I'_v$ is completely adjacent to $C \setminus \set{v}$, then $G[I'_v
	\cup C]$ is a complete multipartite graph, which implies by
	\Cref{lem:2K2char(inftyK)polar} that $G$ is an $(\infty, k)$-polar graph,
	contradicting the election of $G$.
	
	Thus, $I'_v$ is not completely adjacent to $C \setminus \set{v}$, so there is
	a vertex $w \in C \setminus \set{v}$ such that $I'_v$ is completely adjacent
	to $C \setminus \set{v,w}$ and $w$ is completely nonadjacent to $I'_v$.
	Observe that we have two cases depending on whether $w = u$. Since $G[I'_u
	\cup C \setminus \set{u}]$ is a complete multipartite graph and $I'_v
	\subseteq I'_u$, we have that $G[I'_v \cup C \setminus \set{u}]$ is also a
	complete multipartite graph. Then, if $w \ne u$, we have that $v$ is
	adjacent to every vertex of $I'_v$, but this would imply that $G[I'_v \cup
	C]$ is a complete multipartite graph, and we previously noticed that this is
	impossible. Hence $w = u$ and, since $G[I'_v \cup C \setminus \set{u}]$ is a
	complete multipartite graph but $G[I'_v \cup C]$ is not, we have that $v$ is
	completely nonadjacent to $I'_v$, and it follows that the neighborhood of
	each vertex in $I'_v$ equals $C \setminus \set{u,v}$.

	By our previous arguments, there are at least $\lceil c/2 \rceil$ vertices
	of $u \in C$ whose associated sets $I'_u$ are pairwise different. Therefore,
	since $I'_u \subseteq I$, we have that $\lceil C/2 \rceil \le |\mathcal
	P(I)| = 2^{|I|} \le 2^{2k-2}$, from which we conclude that
	\[ |V_G| = |C| + |S| + |I| \le 2 + 2k + 2^{2k-1}. \]
\end{proof}

% % % % % % % % % % % % % % % % % % % % % % % % % % % % % % % % % % % % % % % %
\subsubsection{Algorithms for polarity on $2K_2$-split graphs} \label{ss:alg2K2}
% % % % % % % % % % % % % % % % % % % % % % % % % % % % % % % % % % % % % % % %

We have observed before that $2K_2$-split graphs are unipolar and co-unipolar,
and hence polar graphs. Additionally, we proved that deciding monopolarity and
co-monopolarity in $2K_2$-split graphs can be done in linear time from its
degree sequence. In this section we prove that the problems of deciding whether
a $2K_2$-split graph is $(s, \infty)$-, $(\infty, k)$- or $(s, k)$-polar also
can be efficiently solved.

We start proving that, for any positive integer $s$, the $2K_2$-split graphs
that admit an $(s, \infty)$-polar partition can be recognized in linear time
from their degree sequence.

\begin{theorem} \label{prop:2K2sInftyRecog}
	Let $s$ be an integer, $s \ge 2$. The problem of deciding whether a
	$2K_2$-split graph is $(s, \infty)$-polar is linear-time solvable from its
	degree sequence.
\end{theorem}

\begin{proof}
	Let $G = (C, S, I)$ be a $2K_2$-split graph, and let $c$ and $i$ be the
	number of vertices in $C$ and $I$, respectively. If $c \le s$, then $(C, S
	\cup I)$ is an $(s, \infty)$-polar partition of $G$. Otherwise, we have from
	\Cref{lem:2K2char(sInfty)polar} that $G$ is an $(s, \infty)$-polar graph if
	and only if there exist at least $c-s+2$ vertices of $G$ whose degree is
	exactly $c+3$. By \Cref{theo:charHsplit}, these verifications can be done in
	linear time from the degree sequence of $G$.
\end{proof}

We do not known whether $(\infty, k)$-polarity can be decided in linear time for
$2K_2$-split graphs, but in the next proposition we prove that this problem can
be solved in polynomial time.

\begin{theorem} \label{prop:2K2inftyKRecog}
	Let $k$ be an integer, $k \ge 2$. The problem of deciding whether a
	$2K_2$-split graph is $(\infty, k)$-polar is solvable in quadratic time.
\end{theorem}

\begin{proof}
	Let $G = (C, S, I)$ be a $2K_2$-split graph, and let $c$ and $i$ be the
	number of vertices in $C$ and $I$, respectively. If $i \le k-1$ and $\set{u,
	v}$ is a maximum clique of $G[S]$, then $(C \cup \set{u, v}, I \cup S
	\setminus \set{u,v})$ is an $(\infty, k)$-polar partition of $G$. Else, if
	the subset $I'$ of all vertices of degree $c$ in $G$ has at least $i-k+2$
	elements, then $(C \cup I', S \cup I \setminus I')$ is an $(\infty,
	k)$-polar partition of $G$. Hence, if $i \le k-1$ or there are at most
	$i-k+2$ vertices of degree $c$ in $G$, then $G$ is an $(\infty, k)$-polar
	graph. Now, let us assume that $i \ge k$ and there are at most $i-k+1$
	vertices of $G$ whose degree is $c$.

	For each vertex $v \in C$, let $I^\ast_v$ be the set of all vertices whose
	neighborhood is $C \setminus \set{v}$. It follows from
	\Cref{lem:2K2char(inftyK)polar} that $G$ is an $(\infty, k)$-polar graph if
	and only if $I^\ast_v$ has at least $i-k+2$ vertices for some $v \in C$. The
	result follows since all the verifications can be performed in quadratic
	time.
\end{proof}

As a consequence of \Cref{rem:degrees,theo:DegreeChar(sk)polIPS}, deciding
whether a pseudo-split graph admits an $(s, k)$-polar partition can be done in
linear time from its degree sequence. In contrast, it cannot be decided in
general whether a $2K_2$-split graph is $(s, k)$-polar only from its degree
sequence. For instance, in \Cref{fig:NoFromDS} are depicted two strict
$2K_2$-split graphs with the same degree sequence such that the left one is $(5,
4)$-polar but the right one is not. Despite of that, through the following
propositions we prove that the problem of recognizing $2K_2$-polar graphs that
admit an $(s, k)$-polar partition is solvable in polynomial time.

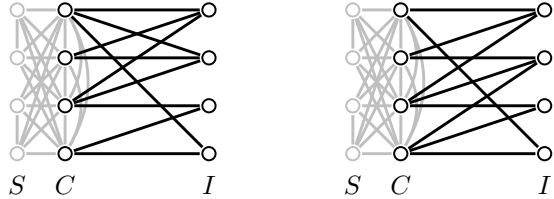
\begin{figure}[!ht]
    \centering
    \begin{tikzpicture}[scale=0.85]

    \begin{scope}[scale=0.75]

    \begin{scope}[xshift=0cm, yshift=0cm]
        \foreach \i in {0,...,3}
            \node (-\i) [vertex, lightgray] at (-1, \i){};
        \foreach \i in {0,...,3}
            \node (0\i) [vertex] at (0, \i){};
        \foreach \i in {0,...,3}
            \node (3\i) [vertex] at (3, \i){};

        % \node (x) [rectangle] at (1,-0.75){$G$};
        \node (x) [rectangle] at (-1,-0.65){$S$};
        \node (x) [rectangle] at (0,-0.65){$C$};
        \node (x) [rectangle] at (3,-0.65){$I$};
    \end{scope}

    \foreach \i in {-0,-1,-2,-3}
        \foreach \j in {00,01,02,03}
            \draw  [edge, lightgray] (\i) to node [above] {} (\j);	

    \draw  [edge, lightgray] (-0) to node [above] {} (-1);	
    \draw  [edge, lightgray] (-2) to node [above] {} (-3);

    \foreach \i/\j in {03/01,03/00,02/00}
        \draw  [bentE, lightgray] (\i) to node [above] {} (\j);
    \foreach \i/\j in {03/02,02/01,01/00}
        \draw  [edge, lightgray] (\i) to node [above] {} (\j);

    \foreach \i in {1,2,3}
        \draw  [edge] (33) to node [above] {} (0\i);	
    \foreach \i in {1,2,3}
        \draw  [edge] (32) to node [above] {} (0\i);	
    \foreach \i in {0,1}
        \draw  [edge] (31) to node [above] {} (0\i);	
    \foreach \i in {0,3}
        \draw  [edge] (30) to node [above] {} (0\i);	

    \begin{scope}[xshift=7cm, yshift=0cm]
        \foreach \i in {0,...,3}
            \node (-\i) [vertex, lightgray] at (-1, \i){};
        \foreach \i in {0,...,3}
            \node (0\i) [vertex] at (0, \i){};
        \foreach \i in {0,...,3}
            \node (3\i) [vertex] at (3, \i){};

        % \node (x) [rectangle] at (1,-0.75){$H$};
        \node (x) [rectangle] at (-1,-0.65){$S$};
        \node (x) [rectangle] at (0,-0.65){$C$};
        \node (x) [rectangle] at (3,-0.65){$I$};
    \end{scope}

    \foreach \i in {-0,-1,-2,-3}
        \foreach \j in {00,01,02,03}
            \draw  [edge, lightgray] (\i) to node [above] {} (\j);	

    \draw  [edge, lightgray] (-0) to node [above] {} (-1);	
    \draw  [edge, lightgray] (-2) to node [above] {} (-3);	

    \foreach \i/\j in {03/01,03/00,02/00}
        \draw  [bentE, lightgray] (\i) to node [above] {} (\j);
    \foreach \i/\j in {03/02,02/01,01/00}
        \draw  [edge, lightgray] (\i) to node [above] {} (\j);

    \foreach \i in {1,2,3}
        \draw  [edge] (33) to node [above] {} (0\i);	
    \foreach \i in {0,1,2}
        \draw  [edge] (32) to node [above] {} (0\i);	
    \foreach \i in {0,1}
        \draw  [edge] (31) to node [above] {} (0\i);	
    \foreach \i in {0,3}
        \draw  [edge] (30) to node [above] {} (0\i);

    \end{scope}

    \end{tikzpicture}

    \caption{Two $2K_2$-split graphs with the same degree sequence such that the
    one on the left side is $(5,4)$-polar but the one on the right side is not.}
    \label{fig:NoFromDS}
\end{figure}

\begin{lemma} \label{lem:kClusterSplit}
	Let $G = (S, K)$ be a split graph, and let $k$ be a positive integer. Let
	$S'$ be the set of all vertices in $S$ which are completely nonadjacent to
	$K$. Then, $G$ is a $k$-cluster if and only if the following sentences hold
	true.
	\begin{enumerate}
		\item For each vertex $w \in S$, either $N(w) = \varnothing$ or $N(w) =
		K$.

		\item $|S \setminus S'| \le 1$.

		\item If $K \ne \varnothing$, then $|S'| \le k-1$. Otherwise $|S'| \le
		k$.
	\end{enumerate}

	Consequently, it can be decided whether a split graph is a $k$-cluster in
	linear time from its degree sequence.
\end{lemma}

\begin{proof}
	The proposition can be easily verified if $K = \varnothing$, so we will assume
	for the proof that $K \ne \varnothing$. Notice that $K \cup S \setminus S'$
	induces a component of $G$ and the other components of $G$ are trivial graphs
	induced by the singletons $\set{w}$ such that $w\in S'$. In consequence, $G$
	has exactly $1 + |S'|$ components. 

	Assume that $G$ is a $k$-cluster. Since $G$ is $P_3$-free, any vertex $w \in
	S$ is either completely adjacent or completely nonadjacent to $K$, and there
	is at most one vertex of $S$ that is not an isolated vertex. Moreover, since
	$K \ne \varnothing$, it follows from our initial observation about the
	components that $|S'| \le k-1$. Therefore, if $G$ is a $k$-cluster, the three
	listed conditions hold. The converse implication follows follows from the
	observations in the first paragraph of this proof and the third statement.
	
	For the last part, suppose that $G$ has degree sequence $d_1 \ge \dots \ge
	d_n$, and let $p = \max \set{i \mid d_i \ge i-1}$. We have from
	\Cref{theo:recognitionSplit} that $(S, K)= (\set{v_{p+1}, \dots,v_n},
	\set{v_1, \dots,v_p})$ is a split partition of $G$ such that $K$ is a maximum
	clique. Then, we have from the characterization above that $G$ is a
	$k$-cluster if and only if $G$ has at most $k$ components and $S$ is
	completely nonadjacent to $K$. For the first condition, notice that $G$ has at
	most $k$ connected components if and only if either $p = 1$ and $n \le k$ or
	$p > 1$ and $n-p \le k-1$, and this can be checked in linear time. The second
	condition is satisfied if and only if $d_1 = \dots = d_p = p-1$, which also
	can be verified in linear time.
\end{proof}

\begin{lemma} \label{lem:(sk)polSplit}
	Let $s$ and $k$ be nonnegative integers such that $s+k \ge 1$. It can be
	decided whether a split graph is $(s, k)$-polar in linear time from its degree
	sequence.
\end{lemma}

\begin{proof}
	Let $G$ be a split graph. Since split graphs are precisely the $(1,1)$-polar
	graphs, if $s$ and $k$ are both positive integers, then $G$ is $(s,k)$-polar.
	Otherwise, $s = 0$ or $k = 0$, and this case follows from
	\Cref{lem:kClusterSplit}.
\end{proof}

As we mentioned before, it is known that $(s,k)$-polar graphs can be recognized
in polinomial time \cite{feder2007matrix,feder2003list}. Next, we present an
alternative polinomial-time algorithm to recognize $(s,k)$-polar graphs on the
class of $2K_2$-split graphs.

\begin{theorem} \label{thm:2K2sKrecog}
	Let $s$ and $k$ be nonnegative integers such that $s+k \ge 1$. Deciding
	whether a $2K_2$-split graph is $(s, k)$-polar can be done in
	polynomial-time.
\end{theorem}

\begin{proof}
	Let $G = (C, S, I)$ be a $2K_2$-split graph, and let $c$ and $i$ be the
	number of vertices in $C$ and $I$, respectively. Let us denote by $I^\ast$
	the set of all vertices of $G$ of degree $c-1$ and, for each vertex $v$ in
	$C$, let $I^\ast_v$ be the set of all vertices $w \in I$ such that $N(w) = C
	\setminus \set{v}$.
		
	From \Cref{lem:(sk)polSplit} we have the result for the case in which
	$S=\varnothing$, so we can assume that $G$ is a strict $2K_2$-split graph.
	In addition we have the following particular cases.
	\begin{enumerate}
		\item $2K_2$ is $(0,2)$- and $(2,1)$-polar but it is neither $(1,1)$-
		nor $(\infty,0)$-polar.

		\item $K_1\oplus 2K_2$ is $(1,2)$- and $(2,1)$-polar but it is not
		$(\infty,0)$-, $(0, \infty)$-, or $(1,1)$-polar.

		\item For $c\ge 2$, $K_c\oplus 2K_2$ is $(2,1)$-polar but it is neither
		$(1, \infty)$- nor $(\infty,0)$-polar.

		\item For $c\ge 1$, $iK_1 + 2K_2$ is $(0,i+2)$- and $(1,2)$-polar but it
		is neither $(0,i+1)$- nor $(\infty,1)$-polar.
	\end{enumerate}
	These cases correspond to the conditions $C = \varnothing$ or $I =
	\varnothing$, that can be checked in linear time from the degree sequence of
	$G$ from \Cref{theo:charHsplit}.

	From the above observations, we can assume for the rest of the proof that
	the sets $C, S$ and $I$ are all of them nonempty. We consider the following
	particular cases.

	\begin{itemize}
		\item If $s, k \le 1$, then $G$ is not $(s, k)$-polar, because $2K_2\le
		G$.

		\item If $k=0$, then $G$ is not a $(s, k)$-polar, because $2K_2\le G$.
			
		\item If $s=0$ and $k \ge 2$, then $G$ is not $(s, k)$-polar, because
		$K_1 \oplus 2K_2 \le G$.

		\item If $k=1$ and $s\ge 2$, then $G$ is not $(s, k)$-polar, because
		$2K_2+K_1 \le G$.

		\item If $s=1$ and $k \ge 2$, we have from \Cref{theo:2K2(1k)obsmin}
		that $G$ is an $(s, k)$-polar graph if and only if $c=1$ and $|\set{w
		\in I : \Deg{w}>0}| \le k-2$. This condition can be verified from
		the degree sequence of $G$ in linear time.
	\end{itemize}

	Notice that, if none of the cases listed before occurs, then $s, k \ge 2$,
	so we can use the characterizations provided by \Cref{thm:2K2Char(sk)pol}.
	The following cases are based on the that characterizations.

	\begin{enumerate}
		\item If $c\le s$ and $i\le k-2$, then $G$ is an $(s, k)$-polar graph.

		\item If $c\le s-2$ and $i\le k-1$, then $G$ is an $(s, k)$-polar graph.

		\item If $c\ge s+1$ and $i\ge k$, then $G$ is not an $(s, k)$-polar graph.

		\item If $c>s$ and $i<k$, then $G$ is an $(s, k)$-polar graph if and only if
		there exist at least $c-s+2$ vertices whose degree is exactly $c+3$. This
		condition can be verified from the degree sequence of $G$ in linear time.

		\item If $c<s$ and $i\ge k$. We can verify from the degree sequence of $G$
		if there exist at least $i-k+2$ vertices of degree exactly $c$; if such
		vertices exist $G$ is an $(s, k)$-polar graph. Otherwise, if
		$|I^\ast|<i-k+2$, $G$ is not an $(s, k)$-polar graph, and if $|I^\ast|\ge
		i-k+2$ we can check, for each vertex $v\in C$, whether the set $I^\ast_v$
		has at least $i-k+2$ vertices; in this point $G$ is an $(s, k)$-polar graph
		if and only if $|I^\ast_v|\ge i-k+2$ for some $v\in C$. These verifications
		can be done in polynomial time.

		\item If $c=s$ and $i\ge k$. If $|I^\ast|<i-k+2$, $G$ is not an $(s,
		k)$-polar graph. Otherwise, if $|I^\ast|\ge i-k+2$, we can check for each
		vertex $v\in C$ whether the set $I^\ast_v$ has at least $i-k+2$ vertices;
		$G$ is an $(s, k)$-polar graph if and only if $|I^\ast_v|\ge i-k+2$ for some
		$v\in C$.

		\item If $c=s$ and $i = k-1$. If there exist at least two vertices of degree
		exactly $c+3$, $G$ is an $(s, k)$-polar graph. Otherwise, if
		$I^\ast=\varnothing$, $G$ is not an $(s, k)$-polar graph, and if $I^\ast\ne
		\varnothing$, we can check for each vertex $v\in C$ whether the set
		$I^\ast_v$ is empty; $G$ is an $(s, k)$-polar graph if and only if
		$I^\ast_v\ne \varnothing$ for some $v\in C$.

		\item If $c=s-1$ and $i = k-1$. First, if there exists a vertex of degree
		$c+3$ or a vertex of degree $c$, $G$ is an $(s, k)$-polar graph. Otherwise,
		if $I^\ast=\varnothing$, $G$ is not an $(s, k)$-polar graph, and if
		$I^\ast\ne \varnothing$, we can check for each vertex $v\in C$ whether the
		set $I^\ast_v$ is empty; $G$ is an $(s, k)$-polar graph if and only if
		$I^\ast_v\ne \varnothing$ for some $v\in C$.
	\end{enumerate}

	The result follows since all verifications can be performed in
	polynomial-time.
\end{proof}

% % % % % % % % % % % % % % % % % % % % % % % % % % % % % % % % % % % % % % % %
\section{Concluding remarks} \label{sec:concl}
% % % % % % % % % % % % % % % % % % % % % % % % % % % % % % % % % % % % % % % %

It is worth noticing that, unlike the upper bound for the order of $2K_2$-split
minimal $(s, \infty)$-polar obstructions provided in
\Cref{theo:uBound2K2sSinfty}, which is linear on $s$, the bound given in
\Cref{theo:uBound2K2sInftyK} for the order of $2K_2$-split minimal $(\infty,
k)$-polar obstructions is exponential on $k$. Moreover, we know that the first
of these bounds is tight, but we cannot guarantee the same for the second one.
We pose the following question.

\begin{problem} \label{prob:improveBound}%
  Can \Cref{lem:2K2obsminBound6,lem:2K2obsminBound5} be improved by replacing
	the condition $|I| \ge 2k-1$ for a stronger one like $|I| \ge k+c_0$, for a
	constant $c_0$?
\end{problem}

Notice that, from the proof used for \Cref{theo:uBound2K2sInftyK}, an
affirmative answer to \Cref{prob:improveBound} would imply an improvement of the
bound provided in the mentioned theorem. Nevertheless, the next observation
makes us think the answer to \Cref{prob:improveBound} could be negative.

\begin{remark}
	For any integers $s$ and $k$, $s, k \ge 2$, the strict $2K_2$-split graph $G =
  (C, S, I)$ with $C = \set{w}$, $I = \set{i_1, \dots, i_{2k-2}}$, and such that
  $wi_j\in E$ if and only if $1 \le j \le k-1$, is a minimal $(s, k)$-polar
  obstruction, and hence it is a minimal $(\infty, k)$-polar obstruction.
\end{remark}

We think that the next question can be answered in an affirmative way by
imitating the proof of \Cref{theo:uBound2K2sSinfty}, which is very different
than the one we used in \Cref{theo:uBound2K2sInftyK}.

\begin{problem}
	Is the order of the $2K_2$-split minimal $(\infty, k)$-polar obstructions
	upper bounded by a function linear on $k$?
\end{problem}

Moreover, some initial explorations allow us to pose the following conjecture.

\begin{conjecture}
	Let $k$ be an integer, $k \ge 3$, and let $G = (C, S, I)$ be a $2K_2$-split
	minimal $(\infty,k)$-polar obstruction. Then $k \le i \le 2k-2$ and $c \le
	2k-i-1$, where $c$ and $i$ stands for the number of vertices in $C$ and $I$,
	respectively.
\end{conjecture}

Notice that, if the previous conjecture is true, then the order of every
$2K_2$-split minimal $(\infty, k)$-polar obstruction does not exceed $2k+3$. In
addition, such a bound would be tight since the strict $2K_2$-split graph with
$C = \set{c_1, \dots, c_{k-1}}$, $I = \set{i_1, \dots, i_k}$ and such that, for
each $j \in \set{1, \dots, k-1}$, $N(c_j) \cap I = I \setminus \set{i_j}$, is a
minimal $(\infty, k)$-polar obstruction for every integer $k \ge 2$.

\end{document}